\documentclass[11pt,twoside]{amsart}
\usepackage[latin9]{inputenc}
\usepackage[a4paper]{geometry}
\usepackage{amstext}
\usepackage{amsthm}
\usepackage{amssymb}
\usepackage{stmaryrd}
\usepackage[unicode=true,pdfusetitle,
 bookmarks=true,bookmarksnumbered=false,bookmarksopen=false,
 breaklinks=false,pdfborder={0 0 1},backref=false,colorlinks=false]
 {hyperref}

\makeatletter
\numberwithin{equation}{section}
\numberwithin{figure}{section}
\theoremstyle{plain}
\newtheorem*{thm*}{\protect\theoremname}
\theoremstyle{plain}
\newtheorem{thm}{\protect\theoremname}
\theoremstyle{plain}
\newtheorem{lem}[thm]{\protect\lemmaname}
\theoremstyle{remark}
\newtheorem{rem}[thm]{\protect\remarkname}
\theoremstyle{definition}
\newtheorem{defn}[thm]{\protect\definitionname}
\theoremstyle{plain}
\newtheorem{cor}[thm]{\protect\corollaryname}
\theoremstyle{definition}
\newtheorem{example}[thm]{\protect\examplename}

\usepackage{amssymb,amsthm,amsmath,amsfonts,amscd}
\usepackage{graphicx}
\usepackage{url}
\usepackage{color}
\usepackage[active]{srcltx}
\usepackage[matrix,arrow]{xy}
\usepackage{mathrsfs}
\usepackage{enumerate}
\usepackage{amsopn} % provides \DeclareMathOperator
\usepackage{bbm} % additional fonts
\usepackage[normalem]{ulem} %underline etc.
\usepackage{cite}
\usepackage{hyperref}
\allowdisplaybreaks[1]
\usepackage[english]{babel}
\usepackage{bbm}
\usepackage{mhchem}
\usepackage{mdef}
\date{\today}
\makeatother

\providecommand{\corollaryname}{Corollary}
\providecommand{\definitionname}{Definition}
\providecommand{\examplename}{Example}
\providecommand{\lemmaname}{Lemma}
\providecommand{\remarkname}{Remark}
\providecommand{\theoremname}{Theorem}

\begin{document}
\title[Optimal regularity for the PME]{Optimal regularity for the porous medium equation}
\begin{abstract}
We prove optimal regularity for solutions to porous media equations in Sobolev spaces, based on velocity averaging techniques. In particular, the obtained regularity is consistent with the optimal regularity in the linear limit. 
\end{abstract}

\thanks{Acknowledgements: The author would like to thank Jonas Sauer for carefully proof-reading an early version of the paper.}
\author{Benjamin~Gess}%{Max Planck Institute for Mathematics in the Sciences, Leipzig}
\address{Benjamin~Gess\newline
Max Planck Institute for Mathematics in the Sciences\newline Inselstr.~22, 04103 Leipzig, Germany\newline
and Fakult\"at f\"ur Mathematik, Universit\"at Bielefeld\newline Universit\"atstr. 25, 33615 Bielefeld, Germany }
\email{bgess@mis.mpg.de}

\maketitle

\section{Introduction}

We establish the optimal spatial regularity of solutions of the porous medium equation 
\begin{align}
\partial_{t}u & =\D(|u|^{m-1}u)\text{ on }(0,T)\times\R_{x}^{d}\label{eq:PME-intro}\\
u(0) & =u_{0}\text{ on }\R_{x}^{d},\nonumber 
\end{align}
with $u_{0}\in L^{1}(\R_{x}^{d})$, $T\ge0$, $m>1$. 

All known regularity estimates in terms of Hölder or Sobolev spaces are restricted to a degree of differentiability of an order less than one. The best known regularity estimate in Sobolev spaces, obtained by Tadmor and Tao in \cite{TT07} and Ebmeyer in \cite{E05}, is that, if $u_{0}\in(L^{1}\cap L^{\infty})(\R_{x}^{d})$, then
\begin{equation}
u\in L^{m+1}([0,T];W_{loc}^{\frac{2}{m+1}-,m+1}(\R_{x}^{d})).\label{eq:eb-reg-2}
\end{equation}
Since $\frac{2}{m+1}<1$ this estimate is inconsistent with the optimal order of differentiability in the linear case of the heat equation ($m=1$) which is $u\in L^{1}([0,T];W^{2,1}(\R_{x}^{d})).$

A scaling argument (cf.~Appendix \ref{app:optimality_scaling} below) shows that it may be possible to improve the regularity to $u\in L^{m}([0,T];\dot{W}^{\frac{2}{m},m}(\R_{x}^{d}))$ which is consistent with the linear case $m=1$. The Barenblatt solution shows that this is the optimal regularity. This is the main result of this paper.
\begin{thm*}
Let $u_{0}\in(L^{1}\cap L^{1+\ve})(\R_{x}^{d})$ for some $\ve>0$. Then, for all $p\in[1,m)$, $s<\frac{2}{m}$,
\begin{equation}
u\in L^{p}([0,T];\dot{W}_{loc}^{s,p}(\R_{x}^{d})).\label{eq:main_est}
\end{equation}
Moreover, there is a constant $C\ge0$ such that 
\[
\|u\|_{L_{t}^{p}\dot{W}_{x,loc}^{s,p}}\le C\left(\|u_{0}\|_{L_{x}^{1}\cap L_{x}^{1+\ve}}^{2}+1\right).
\]
\end{thm*}
The precise statement is given in Theorem \ref{thm:pme} below.

In addition, we treat more general classes of equations, in particular including anisotropic porous media equations of the form
\begin{align}
\partial_{t}u & =\sum_{j=1}^{d}\partial_{x_{j}x_{j}}u^{[m_{j}]}+S(t,x)\text{ on }(0,T)\times\R_{x}^{d},\label{eq:PME-intro-1}
\end{align}
with $u_{0}\in L^{1}(\R_{x}^{d})$, $S\in L^{1}([0,T]\times\R_{x}^{d})$ and $u^{[m]}:=|u|^{m-1}u$. Setting $1<\underline{m}:=\min\{m_{j}\}$, $\overline{m}:=\max\{m_{j}\}$ we obtain that, for all $s<\frac{2}{\overline{m}}\left(\frac{\underline{m}-1}{\overline{m}-1}\right)$, $p<\frac{2\overline{m}}{\overline{m}+1}$, 
\[
\int_{v}f(t,x,v)\phi(v)\,dv\in L^{p}([0,T];W_{loc}^{s,p}(\R_{x}^{d}))
\]
where $f(t,x,v):=1_{v<u(t,x)}-1_{v<0}$ and $\phi$ is an arbitrary cut-off function (see Theorem \ref{ex:anisoptropic_PME} below for details).

In a third main result, we consider the degenerate parabolic Anderson model
\begin{align}
\partial_{t}u & =\partial_{xx}u^{[m]}+u\,S\text{ on }(0,T)\times I\label{eq:intro-anderson}\\
u & =0\text{ on }(0,T)\times\partial I\nonumber \\
u(0) & =u_{0}\in L^{m+1}(I),\nonumber 
\end{align}
on an open, bounded interval $I\subseteq\R$, with $m\in(1,2)$ and $S$ being spatial white noise. The additional difficulty in this case is the irregularity of the source $S$, since spatial white noise is a distribution only. We again obtain regularity consistent with the optimal regularity in the linear case ($m=1$).
\begin{thm*}
Let $u_{0}\in L^{m+1}(I).$ Then there exists a weak solution $u$ to \eqref{eq:intro-anderson} satisfying, for all $p\in[1,m)$, $s<\frac{3}{2}\frac{1}{m}$, 
\begin{equation}
u\in L{}^{p}([0,T];W_{loc}^{s,p}(I)).\label{eq:main_est-1}
\end{equation}
Moreover, there is a constant $C\ge0$ such that 
\begin{align*}
\|u\|_{L_{t}^{p}W_{x,loc}^{s,p}} & \le C(\|u_{0}\|_{L_{x}^{m+1}}^{m+1}+\|S\|_{B_{\infty,\infty}^{-\eta}}^{\tau}+1),
\end{align*}
for some $\tau\ge2$ and $\eta\in(\frac{1}{2},1]$ small enough.
\end{thm*}
The precise statement is given in Corollary \ref{cor:ex-appr-anderson} below. 

The proof presented in this paper is based on Fourier analytic techniques and averaging Lemmata. The first step is to pass to a kinetic formulation of \eqref{eq:PME-intro}. Introducing the kinetic function $f(t,x,v):=1_{v<u(t,x)}-1_{v<0}$ leads to the kinetic form of \eqref{eq:PME-intro}
\begin{align}
\partial_{t}f & =m|v|^{m-1}\Delta f+\partial_{v}q,\label{eq:kinetic_ph-2-1-1}
\end{align}
for some non-negative measure $q$. Since this constitutes a linear equation in $f$, the regularity of velocity averages $\int f\phi(v)\,dv$ for smooth cut-off functions $\phi$ can be analyzed by means of suitable micro-local decompositions in Fourier space. Up to this point our setup is in line with \cite{TT07}. However, in the available literature, one of the drawbacks of analyzing regularity by means of averaging techniques is that it was unknown how to make use of the sign of the measure $q$. Indeed, these arguments were only able to use the fact that the total variation norm of $q$ is finite (cf.~e.g.~\cite{DLOW03,DLW03}). In contrast, in this work, we make use of the additional fact that the entropy dissipation measure $q$ has finite singular moments, meaning that $|v|^{-\g}q$ has finite mass for all $\g\in[0,1)$. In this way we are able to (indirectly) exploit the sign property of $q$ for the first time. 

In addition, classical averaging techniques are restricted to working in $L^{p}$ spaces with $p\in[1,2]$ (cf.~\cite[Averaging Lemma 2.1]{TT07}), which leads to non-optimal integrability exponents. Indeed, because of this in \cite[(4.10)]{TT07} only $W^{\frac{2}{m+1}-,1}$ regularity for solutions to \eqref{eq:PME-intro} could be shown. In order to obtain the optimal integrability exponent $p<m$ we introduce a new concept of isotropic truncation properties for Fourier multipliers. 

A further obstacle in classical averaging arguments is that they rely on a bootstrap technique. However, even if $u$ is smooth, the kinetic function $f$ will only have up to one spatial derivative. Therefore, the standard bootstrap argument is not suited to prove regularity of a higher (than one) order. In the anisotropic case, this difficulty is avoided in the current paper by directly exploiting the $v$-regularity of $f$. In the isotropic case these issues are overcome by introducing the isotropic truncation property mentioned above. In both cases this allows to fully avoid bootstrapping arguments. In order to underline the differences and improvements with regard to \cite{TT07} we follow the notation and structure of \cite{TT07} as far as possible. While, as usual in the theory of averaging techniques, our proof also relies on a micro-local decomposition in Fourier space, the order of decomposition and real-interpolation, the key Lemma \ref{lem:bsc_est}, the bootstrapping argument and the estimation of the entropy dissipation measure proceed differently, as outlined above.

\subsection{Short overview of the literature}

The study of regularity of solutions to porous media equations has a long history and we make no attempt to reproduce a complete account here. In the absence of external forces, the continuity of weak solutions to the porous medium equation has been first shown in general dimension by Caffarelli-Friedman in \cite{CF79}. This result has been subsequently generalized to the case of forced porous media equations by Sacks in \cite{S83,S83b}, based on arguments developed by Cafarelli-Evans in \cite{CE83}. Further generalizations to more general classes of equations have been shown by DiBenedetto \cite{DB83} and Ziemer \cite{Z82}. A detailed account of these developments may be found in Vazquez \cite{V07}. Hölder continuity of solutions to the porous medium equation without force was first obtained by Caffarelli-Friedman \cite{CF80}, see also \cite{V07,W86}, where it is shown that bounded solutions to the porous medium equations are spatially $\a$-Hölder continuous with $\a=\frac{1}{m}\in(0,1)$. We note that in the linear limit $m\downarrow1$ this does not recover the optimal Hölder regularity of the linear case. A generalization to a more general class of degenerate PDE has been obtained by DiBenedetto-Friedman in \cite{DBF85}. In the recent work \cite{M13}, the assumptions on the forcing have recently been relaxed and quantitative estimates are obtained. In particular, it is shown that the Hölder exponent $\a$ is bounded away uniformly from $0$ for $m\downarrow1$. In the nice recent works \cite{BDG13,BDG14} continuity estimates for the porous medium equation and inhomogeneous generalizations thereof with measure valued forcing have been derived.

A particular feature of the porous medium equation ($m>1$) is the effect of finite speed of propagation and thus the occurrence of open interfaces. The regularity of the open interfaces has attracted a lot of attention in the literature, cf.~e.g.~Caffarelli-Friedman \cite{CF80}, Caffarelli-Vazquez-Wolansky \cite{CVW87}, Koch \cite{K99} and the references therein.

In non-forced porous media equations also higher order regularity estimates have been obtained. In one spatial dimension Aronson-Vázquez \cite{AV87} proved eventual $C^{\infty}$ regularity of solutions. For recent progress in the general dimension case see Kienzler-Koch-Vazquez \cite{KKV16}.

In terms of fractional Sobolev regularity of solutions to the porous medium equation less is known. As mentioned above, Ebmeyer \cite{E05} and Tadmor-Tao \cite{TT07} proved for non forced porous media equations that
\begin{equation}
u\in L^{m+1}([0,T];W_{loc}^{s,m+1}),\quad\forall s<\frac{2}{m+1}.\label{eq:eb-reg-1}
\end{equation}
See also Appendix \ref{app:ebm} for a slight improvement of these results. In the recent work \cite{GS16}, Gianazza-Schwarzacher proved higher integrability for nonnegative, local weak solutions to forced porous media equations in terms of a bound on 
\[
\|u^{\frac{m+1}{2}}\|_{L_{loc}^{2+\ve}((0,T);W_{loc}^{1,2+\ve})}
\]
for all $\ve>0$ small enough. In the case of non-forced porous medium equations, Aronson-Benilan type estimates can be used to derive further regularity properties. For example, in \cite[Theorem 8.7]{V07} it has been shown that $\Delta u^{m}\in L_{loc}^{1}((0,\infty);L^{1})$.

Extensions of \cite{TT07} to stochastic parabolic-hyperbolic equations have been considered in \cite{GH16}.

\subsection{Structure of the paper}

In Section \ref{sec:Anisotropic-case} we will consider the case of anisotropic, parabolic-hyperbolic second order PDE. The proof of certain multiplier estimates will be postponed to the Appendix \ref{sec:Truncation-property-and}. In Section \ref{sec:Isotropic-case} we then treat the isotropic case in more detail, in particular introducing the concept of the isotropic truncation property for Fourier multipliers. We will then deduce our main regularity estimates for forced porous media equations. In Section \ref{sec:Degenerate-parabolic-Anderson} we treat the case of the one-dimensional degenerate parabolic Anderson model. A slight improvements of the results obtained by Ebmeyer \cite{E05} will be presented in Appendix \ref{app:ebm}.

\subsection{Notation\label{subsec:Notation}}

For $p\in[1,\infty)$ we let $L^{p}$ be the usual Lebesgue spaces. The space of all locally finite Radon measures is $\mcM$, the subspace of all measures with finite total variation $\mcM_{TV}$. We let $\mcM^{+}\subseteq\mcM$ be the set of all non-negative, locally finite Radon measures and $\mcM_{TV}^{+}=\mcM_{TV}\cap\mcM^{+}$. When convenient we will use the shorthand notation $L_{x}^{1}=L^{1}(\R_{x}^{d})$, $L_{t,x}^{1}=L^{1}([0,T]\times\R_{x}^{d})$. For $p\ge1$ let $p'$ be its conjugate, that is, $\frac{1}{p}+\frac{1}{p'}=1$. We further let $H^{s,p}$ be the fractional Sobolev spaces defined via their Fourier transform, that is, as in \cite[Definition 6.2.2]{G14} and $W^{s,p}$ be the fractional Sobolev-Slobodeckij spaces (cf.~\cite[Section 7.35]{A75}). For ${\displaystyle 1\leqslant p<\infty,\,s\in(0,\infty)\setminus\N}$ and ${\displaystyle f\in W_{loc}^{\lfloor s\rfloor,1}(\R^{d})}$ let $\theta=s-\lfloor s\rfloor\in(0,1)$, define the (homogeneous) Slobodeckij seminorm by
\[
\|f\|{}_{\dot{W}^{s,p}}:=\sup_{|\alpha|=\lfloor s\rfloor}\left(\int_{\R^{d}}\int_{\R^{d}}\frac{|D^{\a}f(x)-D^{\a}f(y)|^{p}}{|x-y|^{\theta p+d}}\,dxdy\right)^{\frac{1}{p}}
\]
and set $\dot{W}^{s,p}:=\{{\displaystyle f\in W_{loc}^{\lfloor s\rfloor,1}(\R^{d})}:\,\|f\|{}_{\dot{W}^{s,p}}<\infty\}$. For ${\displaystyle f\in L_{loc}^{1}(\R^{d})}$ the total variation is given by 
\[
\|f\|_{\dot{BV}}:=\sup\left\{ \int_{\R^{d}}f(x)\div\phi(x)\,\mathrm{d}x\colon\phi\in C_{c}^{1}(\R^{d},\R^{d}),\ \Vert\phi\Vert_{L^{\infty}(\R^{d})}\leq1\right\} 
\]
and we set $\dot{BV}:=\{{\displaystyle f\in L_{loc}^{1}(\R^{d})}:\,\|f\|_{\dot{BV}}<\infty\}$. We follow the notation of \cite{G14,G14-2} and \cite{BCD11}: Let $\mcN^{s,p}(\R^{d})$ be the Nikolskii spaces (cf.~\cite{PKJF13}) and $B_{p,q}^{s}$ Besov spaces (cf.~\cite{G14-2}). We further let $\tilde{L}_{t}^{p}B_{p,q}^{s}=\tilde{L}^{p}([0,T];B_{p,q}^{s}(\R^{d}))$ denote time-space nonhomogeneous Besov spaces as in \cite[Definition 2.67]{BCD11}. We define the discrete increment operator by $\D_{e}^{h}u:=u(x+he)-u(x)$. For results and standard notations in interpolation theory we refer to \cite{BL76}. We let $\S_{+}^{d\times d}$ denote the space of symmetric, non-negative definite matrices. For $b=(b)_{i,j=1\dots d}\in\S_{+}^{d\times d}$ we set $\s=b^{\frac{1}{2}}$, that is, $b_{i,j}=\sum_{k=1}^{d}\sigma_{i,k}\s_{k,j}$. For a locally bounded function $b:\R\to\S_{+}^{d\times d}$ we let $\b_{i,k}$ be such that $\b_{i,k}'(v)=\sigma_{i,k}(v)$. Similarly, for $\psi\in C_{c}^{\infty}(\R_{v})$ we let $\b_{i,j}^{\psi}$ be such that $(\b_{i,k}^{\psi})'(v)=\psi(v)\sigma_{i,k}(v)$. We further introduce the kinetic function
\[
\chi(u,v):=1_{v<u}-1_{v<0}.
\]
Analogously, for a function $u:[0,T]\times\R^{d}\to\R$ we set $f(t,x,v):=\chi(u(t,x),v):=1_{v<u(t,x)}-1_{v<0}.$ We use the short-hand notation $|\xi|\sim2^{j}$ for the set $\{\xi\in\R:\,2^{j-1}\le|\xi|\le2^{j+1}\}$. For $u\in\R$ we set $u^{[m]}:=|u|^{m-1}u$. For two non-negative numbers $a,b\in\R_{+}$ we write $a\lesssim b$ if there exists a constant $C>0$ such that $a\le Cb$.

\section{Anisotropic case\label{sec:Anisotropic-case}}

We consider equations of the form 
\begin{align}
\partial_{t}f(t,x,v)+a(v)\cdot\nabla_{x}f(t,x,v)-\div(b(v)\nabla_{x}f(t,x,v)) & =:\mcL(\partial_{t},\nabla_{x},v)f(t,x,v)\label{eq:kinetic_aniso}\\
 & =g_{0}(t,x,v)+\partial_{v}g_{1}(t,x,v),\nonumber 
\end{align}
where $a:\R\to\R^{d}$, $b:\R\to\S_{+}^{d\times d}$ are $C^{1}$. The operator $\mcL$ is given by its symbol
\begin{equation}
\mcL(i\tau,i\xi,v)=i\tau+ia(v)\cdot\xi-(\xi,b(v)\xi).\label{eq:mcL}
\end{equation}
In this section we will derive regularity estimates for the velocity average, for $\phi\in C_{b}^{\infty}(\R_{v})$,
\[
\bar{f}(t,x):=\int f(t,x,v)\phi(v)\,dv.
\]
These regularity properties are obtained by using a suitable micro-local decomposition of $f$ in Fourier space, which in turn relies on the so-called truncation property satisfied by the multiplier $\mcL$ (cf.~Appendix \ref{sec:Truncation-property-and} below). In contrast to previous results, we will make use of singular moments of $g_{1}$, that is, for $\gamma\in(0,1)$, 
\[
g_{1}(t,x,v)|v|^{-\gamma}\in\begin{cases}
L^{q}(\R_{t}\times\R_{x}^{d}\times\R_{v}), & 1<q\le2\\
\mcM_{TV}(\R_{t}\times\R_{x}^{d}\times\R_{v}), & q=1.
\end{cases}
\]
An additional difficulty arises in the use of bootstrapping arguments. In the theory of averaging Lemmata, optimal regularity estimates are typically obtained by bootstrapping a first non-optimal regularity estimate. This argument, however, can only be applied if the aspired final order of regularity is less than one. Therefore, we have to devise a proof which avoids the use of a bootstrapping argument, which is achieved in Section \ref{sec:Truncation-property-and} by improving a fundamental $L^{p}$ estimate on a class of Fourier multipliers by directly exploiting regularity of $f$ in the velocity direction. 

\subsection{Anisotropic averaging lemma}
\begin{lem}
\label{lem:av}Let $f\in L_{t,x}^{p}(H_{v}^{\s,p})$ for $1<p\le2$, $\s\in(0,1)$ solve, in the sense of distributions,
\begin{equation}
\mcL(\partial_{t},\nabla_{x},v)f(t,x,v)=\D_{x}^{\frac{\eta}{2}}g_{0}(t,x,v)+\partial_{v}\D_{x}^{\frac{\eta_{}}{2}}g_{1}(t,x,v)\text{ on }\R_{t}\times\R_{x}^{d}\times\R_{v}\label{eq:eqn}
\end{equation}
with $g_{i}$ being locally bounded measures satisfying 
\begin{equation}
|g_{0}|(t,x,v)+|g_{1}|(t,x,v)|v|^{-\gamma}\in\begin{cases}
L^{q}(\R_{t}\times\R_{x}^{d}\times\R_{v}), & 1<q\le2\\
\mcM_{TV}(\R_{t}\times\R_{x}^{d}\times\R_{v}), & q=1,
\end{cases}\label{eq:measure_bound}
\end{equation}
for some $\gamma\ge0$, $\eta\ge0$, $1\le q\le p$ and $\mcL(\partial_{t},\nabla_{x},v)$ as in \eqref{eq:kinetic_aniso} with corresponding symbol $\mcL(i\tau,i\xi,v)$ as in \eqref{eq:mcL}. Let $I\subseteq\R$ be a not necessarily finite interval and set 
\[
\o_{\mcL}(J;\d):=\sup_{\tau\in\R,\,\xi\in\R^{d},|\xi|\sim J}|\Omega_{\mcL}(\tau,\xi;\d)|,\quad\Omega_{\mcL}(\tau,\xi;\d)=\{v\in I:|\mcL(i\tau,i\xi,v)|\le\d\},
\]
and suppose that the following non-degeneracy condition holds: There exist $\a\in(0,q'),\b>0$ such that 
\begin{equation}
\o_{\mcL}(J;\d)\lesssim(\frac{\d}{J^{\b}})^{\a}\quad\forall\d\ge1,\ J\ge1.\label{eq:cdt1}
\end{equation}
Moreover, assume that there exist $\l\ge0$ and $\mu\in[0,1]$ such that, $\forall\d\ge1,\ J\ge1$, 
\begin{equation}
\sup_{\tau,|\xi|\sim J}\sup_{v\in\Omega_{\mcL}(\tau,\xi;\d)}|\partial_{v}\mcL(i\tau,i\xi,v)||v|^{\gamma}\lesssim J^{\l}\d^{\mu}\label{eq:cdt2}
\end{equation}
and $\frac{\a\b}{q'}\le\l+\eta$. Then, for all $s\in[0,s^{*})$, $\td p\in[1,p^{*})$, $\phi\in C_{b}^{\infty}(I)$, $T\ge0$ and $\mcO\subset\subset\R^{d}$, there is a $C\ge0$ such that 
\begin{align*}
\|\int f(t,x,v)\phi(v)\,dv\|_{L^{\td p}([0,T];\dot{W}^{s,\td p}(\mcO))} & \le C\big(\|g_{0}\phi\|_{L_{t,x,v}^{q}}+\||v|^{-\gamma}g_{1}\phi\|_{L_{t,x,v}^{q}}+\|g_{1}\phi'\|_{L_{t,x,v}^{q}}\\
 & +\|f\phi\|_{L_{t,x}^{p}(H_{v}^{\s,p})}+\|f\phi\|_{L_{t,x}^{q}L_{v}^{1}}+\|f\phi\|_{L_{t}^{\td p}L_{x,v}^{1}}\big)
\end{align*}
with $s^{*}:=(1-\t)\frac{\a\b}{r}+\t(\frac{\a\b}{q'}-\l-\eta),$ where $\t=\theta_{\a}$ and $p^{*}$ are given by
\[
\t:=\frac{\frac{\a}{r}}{\a(\frac{1}{r}-\frac{1}{q'})+1}\in(0,1),\,\frac{1}{p^{*}}:=\frac{1-\t}{p}+\frac{\t}{q},\,r\in(\frac{p'}{1+\s p'},p']\cap(1,\infty).
\]
\end{lem}

\begin{proof}
Let $\vp_{0}$, $\vp_{1}$ be smooth functions with $\vp_{0}$ supported in $B_{1}(0)$ and $\vp_{1}$ supported in the annulus $\{\xi\in\R^{d}:\,\frac{1}{2}\le|\xi|\le2\}$ and 
\[
\vp_{0}(\xi)+\sum_{j\in\N}\vp_{1}(2^{-j}\xi)=1,\quad\forall\xi\in\R^{d}.
\]
By considering the decomposition $f=f_{0}+f_{1}$ with 
\begin{align*}
f_{0}: & =\mcF_{x}^{-1}[\vp_{0}(\xi)\mcF_{x}f],\,f_{1}:=\sum_{j\in\N}\mcF_{x}^{-1}[\vp_{1}(\frac{\xi}{2^{j}})\mcF_{x}f],
\end{align*}
we may assume without loss of generality that $f$ has Fourier transform supported on $B_{1}(0)^{c}$, since for all $\eta\in[1,\infty)$ 
\begin{equation}
\|\int f_{0}\phi\,dv\|_{L_{t}^{\eta}\dot{W}_{x}^{s,\eta}}\le\|f\phi\|_{L_{t}^{\eta}L_{x,v}^{1}}.\label{eq:low_modes-homog}
\end{equation}

Partially inspired by \cite[Averaging Lemma 2.3]{TT07} we consider a micro-local decomposition of $f$ with regard to the degeneracy of the operator $\mcL(\partial_{t},\nabla_{x},v)$. Let $\psi_{0}$, $\psi_{1}$ be smooth functions with $\psi_{0}$ supported in $B_{1}(0)$ and $\psi_{1}$ supported in the annulus $\{\xi\in\C:\,\frac{1}{2}\le|\xi|\le2\}$ and 
\[
\psi_{0}(\xi)+\sum_{k\in\N}\psi_{1}(2^{-k}\xi)=1,\quad\forall\xi\in\C.
\]
For $\d>0$ to be specified later we write 
\begin{align*}
f & =\psi_{0}\left(\frac{\mcL(\partial_{t},\nabla_{x},v)}{\d}\right)f+\sum_{k\in\N}\psi_{1}\left(\frac{\mcL(\partial_{t},\nabla_{x},v)}{\d2^{k}}\right)f\\
 & =:f^{0}+f^{1},
\end{align*}
where, for $k\in\N\cup\{0\}$,
\[
\psi_{i}\left(\frac{\mcL(\partial_{t},\nabla_{x},v)}{\d2^{k}}\right):=\mcF_{t,x}^{-1}\psi_{i}\left(\frac{\mcL(i\tau,i\xi,v)}{\d2^{k}}\right)\mcF_{t,x}.
\]
Since $f$ solves \eqref{eq:eqn} we have 
\begin{equation}
\mcL(\partial_{t},\nabla_{x},v)f^{1}(t,x,v)=\sum_{k\in\N}\psi_{1}\left(\frac{\mcL(\partial_{t},\nabla_{x},v)}{\d2^{k}}\right)\left(\Delta_{x}^{\frac{\eta}{2}}g_{0}(t,x,v)+\Delta_{x}^{\frac{\eta}{2}}\partial_{v}g_{1}(t,x,v)\right)\label{eq:eqn_fK-3}
\end{equation}
and thus
\begin{align}
f^{1}(t,x,v)= & \sum_{k\in\N}\frac{1}{\d2^{k}}\td\psi_{1}\left(\frac{\mcL(\partial_{t},\nabla_{x},v)}{\d2^{k}}\right)\Delta_{x}^{\frac{\eta}{2}}g_{0}(t,x,v)\label{eq:eqn_fK-1-1}\\
 & +\sum_{k\in\N}\frac{1}{\d2^{k}}\td\psi_{1}\left(\frac{\mcL(\partial_{t},\nabla_{x},v)}{\d2^{k}}\right)\Delta_{x}^{\frac{\eta}{2}}\partial_{v}g_{1}(t,x,v)\nonumber \\
=: & f^{2}(t,x,v)+f^{3}(t,x,v),\nonumber 
\end{align}
where 
\[
\td\psi(z)=\frac{\psi(z)}{z}.
\]

In conclusion, we have arrived at the decomposition
\[
\bar{f}:=\int f\phi\,dv=\int f^{0}\phi\,dv+\int f^{2}\phi\,dv+\int f^{3}\phi\,dv=:\bar{f}^{0}+\bar{f}^{2}+\bar{f}^{3}.
\]
We aim to estimate the regularity of $\bar{f^{0}}$, $\bar{f^{2}}$, $\bar{f^{3}}$ in Besov spaces. Hence, we decompose each $f^{i}$ into Littlewood-Paley pieces with respect to the $x$-variable. Let $\vp_{0}$, $\vp_{1}$ be as above. We set, for $i=0,2,3$, 
\begin{align*}
f_{j}^{i}: & =\mcF_{x}^{-1}[\vp_{1}(\frac{\xi}{2^{j}})\mcF_{x}f^{i}],\quad\text{for }j\in\N.
\end{align*}
Then, since $f^{i}$ has Fourier transform supported on $B_{1}(0)^{c}$, 
\[
f^{i}=\sum_{j\ge1}f_{j}^{i},
\]
where $\hat{f}_{j}^{i}(\tau,\xi,v)$ is supported on frequencies $|\xi|\sim2^{j}$. 

\textit{Step 1:} $f^{0}$

Let $j\in\N$ arbitrary, fixed. Then, by Lemma \ref{lem:bsc_est} for every $r\in(\frac{p'}{1+\s p'},p']\cap(1,\infty)$, 
\begin{align*}
\|\int f_{j}^{0}\phi\,dv\|_{L_{t,x}^{p}} & \lesssim\|f_{j}\phi\|_{L_{t,x}^{p}(H_{v}^{\s,p})}\sup_{\tau,|\xi|\sim2^{j}}|\Omega_{\mcL}(\tau,\xi,\d)|^{\frac{1}{r}}\\
 & \lesssim\|f\phi\|_{L_{t,x}^{p}(H_{v}^{\s,p})}\left(\frac{\d}{(2^{j})^{\b}}\right)^{\frac{\a}{r}}.
\end{align*}
Hence, $\bar{f}^{0}=\int f^{0}\phi\,dv\in\tilde{L}{}_{t}^{p}B_{p,\infty}^{\frac{\a\b}{r}}$ (cf.~\cite[Definition 2.67]{BCD11}) with 
\[
\|\int f^{0}\phi\,dv\|_{\tilde{L}{}_{t}^{p}B_{p,\infty}^{\frac{\a\b}{r}}}\lesssim\delta^{\frac{\a}{r}}\|f\phi\|_{L_{x}^{p}(H_{v}^{\s,p})}.
\]

\textit{Step 2:} $f^{2}$

Let $j\in\N$ arbitrary, fixed. We set
\begin{align*}
f_{j}^{2,k} & =\frac{1}{\d2^{k}}\mcF_{t,x}^{-1}\vp_{1}(\frac{\xi}{2^{j}})\td\psi_{1}\left(\frac{\mcL(i\tau,i\xi,v)}{\d2^{k}}\right)|\xi|^{\eta}\mcF_{t,x}g_{0}(x,v)\\
 & =\frac{1}{\d2^{k}}\mcF_{t,x}^{-1}\td\psi_{1}\left(\frac{\mcL(i\tau,i\xi,v)}{\d2^{k}}\right)|\xi|^{\eta}\mcF_{t,x}g_{0,j}(x,v).
\end{align*}

Hence,
\begin{align*}
\int f_{j}^{2,k}\phi\,dv & =\frac{1}{\d2^{k}}\int\td\psi_{1}\left(\frac{\mcL(\partial_{t},\nabla_{x},v)}{\d2^{k}}\right)\Delta_{x}^{\frac{\eta}{2}}g_{0,j}\phi\,dv
\end{align*}
and, by Lemma \ref{lem:bsc_est} and since $|\xi|^{\eta}$ acts as a constant multiplier of order $(2^{j})^{\eta}$ on $g_{0,j}$, 
\begin{align*}
\|\int f_{j}^{2,k}\phi\,dv\|_{L_{t,x}^{q}}\lesssim & \frac{1}{\d2^{k}}\Big(\|\int\mcF_{t,x}^{-1}\td\psi_{1}\left(\frac{\mcL(i\tau,i\xi,v)}{\d2^{k}}\right)|\xi|^{\eta}\mcF_{t,x}g_{0,j}\phi\,dv\|_{L_{t,x}^{q}}\\
\lesssim & \frac{\sup_{\tau,|\xi|\sim2^{j}}|\Omega(\tau,\xi,\d2^{k})|^{\frac{1}{q'}}}{\d2^{k}}(2^{j})^{\eta}\|g_{0,j}\phi\|_{L_{t,x,v}^{q}}\\
\lesssim & \frac{1}{\d2^{k}}\left(\frac{\d2^{k}}{(2^{j})^{\b}}\right)^{\frac{\a}{q'}}(2^{j})^{\eta}\|g_{0,j}\phi\|_{L_{t,x,v}^{q}}.
\end{align*}
Hence,
\begin{align*}
\|\int f_{j}^{2}\phi\,dv\|_{L_{t,x}^{q}} & \lesssim\sum_{k\in\N}\frac{1}{\d2^{k}}\left(\frac{\d2^{k}}{(2^{j})^{\b}}\right)^{\frac{\a}{q'}}(2^{j})^{\eta}\|g_{0,j}\phi\|_{L_{t,x,v}^{q}}\\
 & \lesssim\d^{\frac{\a}{q'}-1}(2^{j})^{\eta-\frac{\a\b}{q'}}\|g_{0,j}\phi\|_{L_{t,x,v}^{q}}.
\end{align*}
In conclusion, $\int f^{2}\phi\,dv\in\tilde{L}{}_{t}^{q}B_{q,\infty}^{\frac{\a\b}{q'}-\eta}$ with 
\[
\|\int f^{2}\phi\,dv\|_{\tilde{L}{}_{t}^{q}B_{q,\infty}^{\frac{\a\b}{q'}-\eta}}\lesssim\d^{\frac{\a}{q'}-1}\|g_{0}\phi\|_{L_{t,x,v}^{q}}.
\]

\textit{Step 3:} $f^{3}$

Let $j\in\N$ arbitrary, fixed. We set
\begin{align*}
f_{j}^{3,k} & =\frac{1}{\d2^{k}}\mcF_{t,x}^{-1}\vp_{1}(\frac{\xi}{2^{j}})\td\psi_{1}\left(\frac{\mcL(i\tau,i\xi,v)}{\d2^{k}}\right)|\xi|^{\eta}\mcF_{t,x}\partial_{v}g_{1}(t,x,v)\\
 & =\frac{1}{\d2^{k}}\mcF_{t,x}^{-1}\td\psi_{1}\left(\frac{\mcL(i\tau,i\xi,v)}{\d2^{k}}\right)|\xi|^{\eta}\mcF_{t,x}\partial_{v}g_{1,j}(t,x,v).
\end{align*}

Hence,
\begin{align*}
\int f_{j}^{3,k}\phi\,dv & =\frac{1}{\d2^{k}}\int\td\psi_{1}\left(\frac{\mcL(\partial_{t},\nabla_{x},v)}{\d2^{k}}\right)\Delta_{x}^{\frac{\eta}{2}}\phi\partial_{v}g_{1,j}dv.
\end{align*}

We observe 
\begin{align*}
\int f_{j}^{3,k}\phi\,dv & =-\frac{1}{\d2^{k}}\int\partial_{v}\td\psi_{1}\left(\frac{\mcL(\partial_{t},\nabla_{x},v)}{\d2^{k}}\right)\Delta_{x}^{\frac{\eta}{2}}g_{1,j}\phi\,dv\\
 & -\frac{1}{\d2^{k}}\int\td\psi_{1}\left(\frac{\mcL(\partial_{t},\nabla_{x},v)}{\d2^{k}}\right)\Delta_{x}^{\frac{\eta}{2}}g_{1,j}\phi'\,dv\\
 & =-\frac{1}{\d2^{k}}\int\mcF_{t,x}^{-1}\Big(\td\psi{}_{1}^{'}\left(\frac{\mcL(i\tau,i\xi,v)}{\d2^{k}}\right)\frac{\partial_{v}\mcL(i\tau,i\xi,v)}{\d2^{k}}|\xi|^{\eta}\mcF_{t,x}g_{1,j}\phi\Big)dv\\
 & -\frac{1}{\d2^{k}}\int\td\psi_{1}\left(\frac{\mcL(\partial_{t},\nabla_{x},v)}{\d2^{k}}\right)\Delta_{x}^{\frac{\eta}{2}}g_{1,j}\phi'\,dv\\
 & =\frac{1}{(\d2^{k})^{2}}\int\mcF_{t,x}^{-1}\Big(\td\psi{}_{1}^{'}\left(\frac{\mcL(i\tau,i\xi,v)}{\d2^{k}}\right)\partial_{v}\mcL(i\tau,i\xi,v)|v|^{\gamma}|\xi|^{\eta}\mcF_{t,x}|v|^{-\gamma}g_{1,j}\phi\Big)dv\\
 & -\frac{1}{\d2^{k}}\int\td\psi_{1}\left(\frac{\mcL(\partial_{t},\nabla_{x},v)}{\d2^{k}}\right)\Delta_{x}^{\frac{\eta}{2}}g_{1,j}\phi'\,dv.
\end{align*}
By the Marcinkiewicz Multiplier Theorem (cf.~\cite[Theorem 5.2.4]{G14-2}) and \eqref{eq:cdt2} we have that $\partial_{v}\mcL(i\tau,i\xi,v)|v|^{\gamma}$ acts as a constant multiplier on $L^{q}$ of order $O((2^{j})^{\l}(\d2^{k})^{\mu})$ on $g_{1,j}$. Hence, using Lemma \ref{lem:bsc_est} yields
\begin{align*}
 & \|\int f_{j}^{3,k}\phi\,dv\|_{L_{t,x}^{q}}\\
 & \le\frac{1}{(\d2^{k})^{2}}\|\int\mcF_{t,x}^{-1}\td\psi{}_{1}^{'}\left(\frac{\mcL(i\tau,i\xi,v)}{\d2^{k}}\right)(\partial_{v}\mcL)(i\tau,i\xi,v)|v|^{\gamma}|\xi|^{\eta}\mcF_{t,x}|v|^{-\gamma}g_{1,j}\phi\,dv\|_{L_{t,x}^{q}}\\
 & +\frac{1}{\d2^{k}}\|\int\td\psi_{1}\left(\frac{\mcL(\partial_{t},\nabla_{x},v)}{\d2^{k}}\right)\Delta_{x}^{\frac{\eta}{2}}g_{1,j}\phi'\,dv\|_{L_{t,x}^{q}}\\
 & \lesssim\frac{\sup_{\tau,|\xi|\sim J}|\Omega(\tau,\xi,\d2^{k})|^{\frac{1}{q'}}}{(\d2^{k})^{2}}(2^{j})^{\eta+\l}(\d2^{k})^{\mu}\||v|^{-\gamma}g_{1,j}\phi\,dv\|_{L_{t,x,v}^{q}}\\
 & +\frac{\sup_{\tau,|\xi|\sim J}|\Omega(\tau,\xi,\d2^{k})|^{\frac{1}{q'}}}{\d2^{k}}(2^{j})^{\eta}\|g_{1,j}\phi'\|_{L_{t,x,v}^{q}}\\
 & \lesssim\frac{1}{(\d2^{k})^{2}}\left(\frac{\d2^{k}}{(2^{j})^{\b}}\right)^{\frac{\a}{q'}}(2^{j})^{\eta+\l}(\d2^{k})^{\mu}\||v|^{-\gamma}g_{1,j}\phi\|_{L_{t,x,v}^{q}}\\
 & +\frac{1}{\d2^{k}}\left(\frac{\d2^{k}}{(2^{j})^{\b}}\right)^{\frac{\a}{q'}}(2^{j})^{\eta}\|g_{1,j}\phi'\|_{L_{t,x,v}^{q}}\\
 & =(\d2^{k})^{-2+\frac{\a}{q'}+\mu}(2^{j})^{\eta+\l-\frac{\a\b}{q'}}\||v|^{-\gamma}g_{1,j}\phi\|_{L_{t,x,v}^{q}}+(\d2^{k})^{-1+\frac{\a}{q'}}(2^{j})^{\frac{\a\b}{q'}+\eta}\|g_{1,j}\phi'\|_{L_{t,x,v}^{q}}.
\end{align*}
Hence, for $\d\ge1$ and using $\mu\in[0,1]$, $\a<q'$, 
\begin{align*}
\|\int f_{j}^{3}\phi\,dv\|_{L_{t,x}^{q}}\lesssim & \sum_{k\in\N}(\d2^{k})^{-2+\frac{\a}{q'}+\mu}(2^{j})^{\eta+\l-\frac{\a\b}{q'}}\||v|^{-\gamma}g_{1,j}\phi\|_{L_{t,x,v}^{q}}\\
 & +(\d2^{k})^{-1+\frac{\a}{q'}}(2^{j})^{\frac{\a\b}{q'}+\eta}\|g_{1,j}\phi'\|_{L_{t,x,v}^{q}}\\
\lesssim & \d{}^{-2+\frac{\a}{q'}+\mu}(2^{j})^{\eta+\l-\frac{\a\b}{q'}}\||v|^{-\gamma}g_{1,j}\phi\|_{L_{t,x,v}^{q}}+\d{}^{-1+\frac{\a}{q'}}(2^{j})^{\frac{\a\b}{q'}+\eta}\|g_{1,j}\phi'\|_{L_{t,x,v}^{q}}\\
\lesssim & \d{}^{-1+\frac{\a}{q'}}(2^{j})^{\eta+\l-\frac{\a\b}{q'}}(\||v|^{-\gamma}g_{1,j}\phi\|_{L_{t,x,v}^{q}}+\|g_{1,j}\phi'\|_{L_{t,x,v}^{q}}).
\end{align*}
In conclusion, $\int f^{3}\phi\,dv\in\tilde{L}{}_{t}^{q}B_{q,\infty}^{\frac{\a\b}{q'}-\l-\eta}$ with 
\[
\|\int f^{3}\phi\,dv\|_{\tilde{L}{}_{t}^{q}B_{q,\infty}^{\frac{\a\b}{q'}-\l-\eta}}\lesssim\d{}^{-1+\frac{\a}{q'}}(\||v|^{-\gamma}g_{1}\phi\|_{L_{t,x,v}^{q}}+\|g_{1}\phi'\|_{L_{t,x,v}^{q}}).
\]

\textit{Step 4:} Conclusion

Since $B_{q,\infty}^{\frac{\a\b}{q'}-\eta}\hookrightarrow B_{q,\infty}^{\frac{\a\b}{q'}-\l-\eta}$ we have
\[
\bar{f}=\bar{f}^{0}+\bar{f}^{1}
\]
with $\bar{f}^{0}\in\td L_{t}^{p}B_{p,\infty}^{\frac{\a\b}{r}}$, $\bar{f}^{1}=\bar{f}^{2}+\bar{f}^{3}\in\td L_{t}^{q}B_{q,\infty}^{\frac{\a\b}{q'}-\l-\eta}$ and, for $\d\ge1$,
\begin{align*}
\|\bar{f}^{0}\|_{\tilde{L}{}_{t}^{p}B_{p,\infty}^{\frac{\a\b}{r}}} & \lesssim\delta^{\frac{\a}{r}}\|f\phi\|_{L_{t,x}^{p}(H_{v}^{\s,p})},\\
\|\bar{f}^{1}\|_{\tilde{L}{}_{t}^{q}B_{q,\infty}^{\frac{\a\b}{q'}-\l-\eta}} & \lesssim\d^{\frac{\a}{q'}-1}(\|g_{0}\phi\|_{L_{t,x,v}^{q}}+\||v|^{-\gamma}g_{1}\phi\|_{L_{t,x,v}^{q}}+\|g_{1}\phi'\|_{L_{t,x,v}^{q}}).
\end{align*}
We aim to conclude by real interpolation. We set, for $z>0$,
\begin{align*}
K(z,\overline{f}):=\inf\{ & \|\overline{f}^{1}\|_{\tilde{L}{}_{t}^{q}B_{q,\infty}^{\frac{\a\b}{q'}-\l-\eta}}+z\|\overline{f}^{0}\|_{\tilde{L}{}_{t}^{p}B_{p,\infty}^{\frac{\a\b}{r}}}:\overline{f}^{0}\in\tilde{L}{}_{t}^{p}B_{p,\infty}^{\frac{\a\b}{r}},\\
 & \overline{f}^{1}\in\tilde{L}{}_{t}^{q}B_{q,\infty}^{\frac{\a\b}{q'}-\l-\eta},\ \overline{f}=\overline{f}^{0}+\overline{f}^{1}\}.
\end{align*}
We first note the trivial estimate, since $\frac{\a\b}{q'}-\l-\eta\le0$, 
\[
K(z,\overline{f})\le\|\overline{f}\|_{\tilde{L}{}_{t}^{q}B_{q,\infty}^{\frac{\a\b}{q'}-\l-\eta}}\le\|\overline{f}\|_{\tilde{L}{}_{t}^{q}L_{x}^{q}}\le\|f\phi\|_{L_{t,x}^{q}L_{v}^{1}}\quad\forall z>0.
\]
Hence, it is enough to consider $z\le1$ in the estimates below. By the above estimates we obtain that, for $\d\ge1$, 
\begin{align*}
K(z,\overline{f}) & \le\d^{\frac{\a}{q'}-1}(\|g_{0}\phi\|_{L_{t,x,v}^{q}}+\||v|^{-\gamma}g_{1}\phi\|_{L_{t,x,v}^{q}}+\|g_{1}\phi'\|_{L_{t,x,v}^{q}})+z\delta^{\frac{\a}{r}}\|f\phi\|_{L_{t,x}^{p}(H_{v}^{\s,p})}.
\end{align*}
We now equilibrate the first and the second term on the right hand side, that is, we set 
\[
\d^{\frac{\a}{q'}-1}=z\delta^{\frac{\a}{r}},
\]
which yields
\[
\d=z^{-\frac{1}{\a(\frac{1}{r}-\frac{1}{q'})+1}}\ge1.
\]
Hence, with 
\[
\t=\frac{1-\frac{\a}{q'}}{\a(\frac{1}{r}-\frac{1}{q'})+1}=1-\frac{\frac{\a}{r}}{\a(\frac{1}{r}-\frac{1}{q'})+1}
\]
we obtain, for $|z|\le1$, 
\begin{align*}
K(z,\overline{f}) & \le z^{\t}(\|g_{0}\phi\|_{L_{t,x,v}^{q}}+\||v|^{-\gamma}g_{1}\phi\|_{L_{t,x,v}^{q}}+\|g_{1}\phi'\|_{L_{t,x,v}^{q}}+\|f\phi\|_{L_{t,x}^{p}(H_{v}^{\s,p})}).
\end{align*}
Note that $\t\in(0,1)$ since $\a<q'$. Consequently, for $\tau\in(0,\t)$ and $\frac{1}{p_{\tau}}=\frac{1-\tau}{q}+\frac{\tau}{p}$,
\begin{align*}
\|\overline{f}\|_{(\tilde{L}{}_{t}^{q}B_{q,\infty}^{\frac{\a\b}{q'}-\l-\eta},\tilde{L}{}_{t}^{p}B_{p,\infty}^{\frac{\a\b}{r}})_{\tau,p_{\tau}}}^{p_{\tau}} & =\|z^{-\tau}K(z,\overline{f})\|_{L_{*}^{p_{\tau}}(0,\infty)}^{p_{\tau}}\\
 & =\|z^{-\tau}K(z,\overline{f})\|_{L_{*}^{p_{\tau}}(0,1)}^{p_{\tau}}+\|z^{-\tau}K(z,\overline{f})\|_{L_{*}^{p_{\tau}}(1,\infty)}^{p_{\tau}}\\
 & \le\|z^{\t-\tau}\|_{L_{*}^{p_{\tau}}(0,1)}^{p_{\tau}}\big(\|g_{0}\phi\|_{L_{t,x,v}^{q}}+\||v|^{-\gamma}g_{1}\phi\|_{L_{t,x,v}^{q}}\\
 & +\|g_{1}\phi'\|_{L_{t,x,v}^{q}}+\|f\phi\|_{L_{t,x}^{p}(H_{v}^{\s,p})}\big)^{p_{\tau}}+\|z^{-\tau}\|_{L_{*}^{p_{\tau}}(1,\infty)}^{p_{\tau}}\|f\phi\|_{L_{t,x}^{q}L_{v}^{1}}^{p_{\tau}}\\
 & \lesssim\|g_{0}\phi\|_{L_{t,x,v}^{q}}^{p_{\tau}}+\||v|^{-\gamma}g_{1}\phi\|_{L_{t,x,v}^{q}}^{p_{\tau}}+\|g_{1}\phi'\|_{L_{t,x,v}^{q}}^{p_{\tau}}\\
 & +\|f\phi\|_{L_{t,x}^{p}(H_{v}^{\s,p})}^{p_{\tau}}+\|f\phi\|_{L_{t,x}^{q}L_{v}^{1}}^{p_{\tau}}.
\end{align*}
Let
\begin{align*}
s<s^{*} & :=(1-\t)(\frac{\a\b}{q'}-\l-\eta)+\t\frac{\a\b}{r}
\end{align*}
From \cite[p. 98]{BCD11} we recall, for $\ve>0$, 
\begin{align*}
\tilde{L}{}_{t}^{q}B_{q,\infty}^{\frac{\a\b}{q'}-\l-\eta} & \hookrightarrow\tilde{L}{}_{t}^{q}B_{q,1}^{\frac{\a\b}{q'}-\l-\eta-\ve}\hookrightarrow L{}_{t}^{q}B_{q,1}^{\frac{\a\b}{q'}-\l-\eta-\ve}
\end{align*}
and analogously for $\tilde{L}{}_{t}^{p}B_{p,\infty}^{\frac{\a\b}{r}}$. Thus, using \cite[Section 5.6 and Theorem 6.4.5]{BL76} and choosing $\ve>0$ small enough yields 
\begin{align*}
(\tilde{L}{}_{t}^{q}B_{q,\infty}^{\frac{\a\b}{q'}-\l-\eta},\tilde{L}{}_{t}^{p}B_{p,\infty}^{\frac{\a\b}{r}})_{\tau,p_{\tau}}\hookrightarrow & L^{p_{\tau}}(B_{q,1}^{\frac{\a\b}{q'}-\l-\eta-\ve},B_{p,1}^{\frac{\a\b}{r}-\ve})_{\tau,p_{\tau}}\\
\hookrightarrow & L_{t}^{p_{\tau}}B_{p_{\tau},p_{\tau}}^{s}\hookrightarrow L{}_{t}^{p_{\tau}}W_{x}^{s,p_{\tau}}.
\end{align*}
Hence, choosing $\tau\in(0,\t)$ large enough and recalling \eqref{eq:low_modes-homog}, for all $p<p^{*}$ with $\frac{1}{p^{*}}=\frac{1-\t}{q}+\frac{\t}{p}$ and all $\mcO\subset\R^{d}$ compact, we have 
\begin{align*}
\|\bar{f}\|_{L^{p}([0,T];\dot{W}^{s,p}(\mcO))} & \lesssim\|g_{0}\phi\|_{L_{t,x,v}^{q}}+\||v|^{-\gamma}g_{1}\phi\|_{L_{t,x,v}^{q}}+\|g_{1}\phi'\|_{L_{t,x,v}^{q}}\\
 & +\|f\phi\|_{L_{t,x}^{p}(H_{v}^{\s,p})}+\|f\phi\|_{L_{t,x}^{q}L_{v}^{1}}+\|f\phi\|_{L_{t}^{p}L_{x,v}^{1}}.
\end{align*}
\end{proof}
\begin{rem}
\label{rem:v-loc}In the above averaging Lemma we do not require $\phi$ to have compact support, nor $I$ to be a bounded interval. We note that if $I$ and $\supp\phi$ are unbounded, then the non-degeneracy condition \eqref{eq:cdt1} entails a growth condition on $\mcL(i\tau,i\xi,v)$. 

This becomes clear when looking at specific examples, such as porous media equations with nonlinearity $B(u)$, which in kinetic form corresponds to \eqref{eq:kinetic_aniso} with $a\equiv0$, $b(v)=B'(v)Id$. In this case, $|\mcL(i\tau,i\xi,v)|\ge|\xi|^{2}b(v)$ and thus
\begin{align*}
\o_{\mcL}(J;\d) & =\sup_{\tau,|\xi|\sim J}|\{v\in\supp\phi:|\mcL(i\tau,i\xi,v)|\le\d\}|\\
 & \le\sup_{|\xi|\sim J}|\{v\in\supp\phi:|b(v)|\le\d|\xi|^{-2}\}|\le|b^{-1}(B_{\d|J|^{-2}}(0))\cap\supp\phi|.
\end{align*}
Hence, in the case $\supp\phi=\R$ condition \eqref{eq:cdt1} becomes, roughly speaking, $|b^{-1}(B_{r}(0))|\lesssim r^{\a}$ for all $r>0$.
\end{rem}

\subsection{Anisotropic parabolic-hyperbolic equations}

In this section we consider parabolic-hyperbolic equations of the type
\begin{align}
\partial_{t}u+\div A(u) & =\div(b(u)\nabla u)+S(t,x)\quad\text{on }(0,T)\times\R_{x}^{d}\label{eq:par-hyp}\\
u(0) & =u_{0}\quad\text{on }\R_{x}^{d},\nonumber 
\end{align}
where
\begin{align}
u_{0} & \in L^{1}(\R_{x}^{d}),\,S\in L^{1}([0,T]\times\R_{x}^{d}),\,T\ge0,\nonumber \\
a:=A' & \in C(\R;\R^{d})\cap C^{1}(\R\setminus\{0\};\R^{d}),\label{eq:ph-as}\\
b=(b_{jk})_{j,k=1\dots d} & \in C(\R;S_{+}^{d\times d})\cap C^{1}(\R\setminus\{0\};S_{+}^{d\times d}).\nonumber 
\end{align}
The corresponding kinetic form for 
\begin{equation}
f(t,x,v)=\chi(u(t,x),v)\label{eq:par_hyp_char_fctn}
\end{equation}
reads (cf.~\cite{CP03})  
\begin{align}
\mcL(\partial_{t},\nabla_{x},v)f(t,x,v) & =\partial_{t}f+a(v)\cdot\nabla_{x}f-\div(b(v)\nabla_{x}f)\label{eq:kinetic_ph-3}\\
 & =\partial_{v}q+S(t,x)\d_{u(t,x)=v}(v),\nonumber 
\end{align}
where $q\in\mcM^{+}$ and $\mcL$ is identified with the symbol 
\begin{equation}
\mcL(i\tau,i\xi,v):=i\tau+a(v)\cdot i\xi-(b(v)\xi,\xi).\label{eq:ph-symbol}
\end{equation}
We will use the terms kinetic and entropy solution synonymously. From \cite{CP03} we recall the definition of entropy/kinetic solutions to \eqref{eq:par-hyp}.
\begin{defn}
\label{def:kinetic_sol-1}We say that $u\in C([0,T];L^{1}(\R^{d}))$ is an entropy solution to \eqref{eq:par-hyp} if $f=\chi(u)$ satisfies

\begin{enumerate}
\item For any non-negative $\psi\in\mcD(\R)$, $k=1,\dots,d$,
\[
\sum_{i=1}^{d}\partial_{x_{i}}\b_{ik}^{\psi}(u)\in L^{2}([0,T]\times\R^{d}).
\]
\item For any two non-negative functions $\psi_{1},\psi_{2}\in\mcD(\R)$, $k=1,\dots,d$,
\[
\sqrt{\psi_{1}(u(t,x))}\sum_{i=1}^{d}\partial_{x_{i}}\b_{ik}^{\psi_{2}}(u(t,x))=\sum_{i=1}^{d}\partial_{x_{i}}\b_{ik}^{\psi_{1}\psi_{2}}(u(t,x))\quad\text{a.e.}.
\]
\item There are non-negative measures $m,n\in\mcM^{+}$ such that, in the sense of distributions,
\[
\partial_{t}f+a(v)\cdot\nabla_{x}f-\div(b(v)\nabla_{x}f)=\partial_{v}(m+n)+\delta_{v=u(t,x)}S\quad\text{on }(0,T)\times\R_{x}^{d}\times\R_{v}
\]
where $n$ is defined by
\[
\int\psi(v)n(t,x,v)\,dv=\sum_{k=1}^{d}\left(\sum_{i=1}^{d}\partial_{x_{i}}\b_{ik}^{\psi}(u(t,x))\right)^{2}
\]
for any $\psi\in\mcD(\R)$ with $\psi\ge0.$
\item We have
\[
\int(m+n)\,dxdt\le\mu(v)\in L_{0}^{\infty}(\R),
\]
where $L_{0}^{\infty}$ is the space of $L^{\infty}$-functions vanishing for $|v|\to\infty$.
\end{enumerate}
\end{defn}

 A sketch of the proof of well-posedness of entropy solutions is given in Appendix \ref{app:kin_solutions} below. For notational convenience we set $q=m+n$ in the following. We first establish the following a-priori bound 
\begin{lem}
\label{lem:ph-est}Let $u$ be the unique entropy solution to \eqref{eq:par-hyp} with $u_{0}\in(L^{1}\cap L^{2-\gamma})(\R_{x}^{d})$, $S\in(L^{1}\cap L^{2-\gamma})([0,T]\times\R_{x}^{d})$ for some $\gamma\in(-\infty,1)$. Then, there is a constant $C=C(T,\g)\ge0$ such that
\begin{align}
 & \sup_{t\in[0,T]}\|u(t)\|_{L_{x}^{2-\gamma}}^{2-\gamma}+(1-\gamma)\int_{0}^{T}\int_{\R^{d+1}}|v|^{-\gamma}q\,dvdxdr\label{eq:par-hyp-energy-bound}\\
 & \le C\big(\|u_{0}\|_{L_{x}^{2-\gamma}}^{2-\gamma}+\|S\|_{L_{t,x}^{2-\gamma}}^{2-\gamma}\big).\nonumber 
\end{align}
Moreover, for $\eta\in C_{c}^{\infty}(\R_{v})$, $t\in[0,T]$ we have
\begin{align}
 & \int_{\R_{x}^{d}}\eta(u(t))dx+\int_{0}^{t}\int_{\R^{d+1}}\eta''(v)q\,dvdxdr\label{eq:par-hyp-energy-bound-1}\\
 & \le\int_{\R_{x}^{d}}\eta(u_{0})dx+\|\eta'\|_{_{L_{v}^{\infty}}}\|S\|_{L_{t,x}^{1}}.\nonumber 
\end{align}
\end{lem}

\begin{proof}
By \eqref{eq:lp-bound} below we have $u\in L^{\infty}([0,T];L^{2-\gamma}(\R_{x}^{d}))$. Let $\eta\in C_{c}^{\infty}(\R_{v})$, $\eta(0)=0$, $t\in[0,T]$ and let $\vp^{n}\in C_{c}^{\infty}((0,T)\times\R_{x}^{d})$ be a sequence of cut-off functions satisfying $\vp^{n}=1$ on $(\frac{1}{n},t-\frac{1}{n})\times B_{n}(0)$. Testing \eqref{eq:kinetic_ph-3} by $\eta'(v)\vp^{n}(r,x)$ yields
\begin{align*}
-\int_{0}^{t}\int\eta(u)\partial_{r}\vp^{n}dxdr & =\int_{0}^{t}\int\eta'(v)a(v)f\cdot\nabla_{x}\vp^{n}+\sum_{i,j=1}^{d}\eta'(v)b_{ij}(v)f\partial_{x_{i}x_{j}}\vp^{n}\,dvdxdr\\
 & +\int_{0}^{t}\int\eta'(u)\vp^{n}S\,dxdr-\int_{0}^{t}\int\eta''(v)\vp^{n}q\,dvdxdr.
\end{align*}
Taking the limit $n\to\infty$ yields 
\begin{align*}
\int\eta(u(t))dx & =\int\eta(u_{0})dx+\int_{0}^{t}\int\eta'(u)S\,dxdr-\int_{0}^{t}\int\eta''(v)q\,dvdxdr
\end{align*}
and thus \eqref{eq:par-hyp-energy-bound-1}. Hölder's inequality implies 
\begin{align*}
\int\eta(u(t))dx & \lesssim\int\eta(u_{0})dx+\int_{0}^{t}\int|\eta'(u)|^{\frac{2-\gamma}{1-\gamma}}dxdr+\int_{0}^{t}\int|S|^{2-\gamma}dxdr\\
 & -\int_{0}^{t}\int\eta''(v)q\,dvdxdr.
\end{align*}
Using a standard cut-off argument we may choose $\eta=\eta^{\d}\in C^{\infty}$ with
\[
(\eta^{\d})''(v):=(|v|^{2}+\d)^{-\frac{\gamma}{2}}.
\]
Then $\eta^{\d}$ is convex and $(\eta^{\d})'(v)\le|v|^{1-\gamma}$. Hence, 
\begin{align*}
\int\eta^{\d}(u(t))dx+\int_{0}^{t}\int(\eta^{\d})''(v)q\,dvdxdr\lesssim & \int\eta^{\d}(u_{0})dx+\int_{0}^{t}\int|u|^{2-\gamma}dxdr\\
 & +\int_{0}^{t}\int|S|^{2-\gamma}dxdr.
\end{align*}
Letting $\d\to0$ yields, by Fatou's Lemma,
\begin{align*}
\int|u(t)|^{2-\gamma}dx+\int_{0}^{t}\int|v|^{-\gamma}q\,dvdxdr\lesssim & \int|u_{0}|^{2-\gamma}dx+\int_{0}^{t}\int|u|^{2-\gamma}dxdr\\
 & +\int_{0}^{t}\int|S|^{2-\gamma}dxdr.
\end{align*}
Gronwall's inequality concludes the proof. 
\end{proof}
\begin{lem}
\label{lem:ph-est-1}Let $u$ be the unique entropy solution to \eqref{eq:par-hyp} and $\psi\in C^{2}(\R)\cap Lip(\R)$ be a convex function with $|\psi(r)|\le c|r|$, for some $c>0$. Then 
\[
\int q(t,x,v)\psi''(v)dvdxdt\le C(\|u_{0}\|_{L_{x}^{1}}+\|S\|_{L_{t,x}^{1}}),
\]
for some constant $C$ depending only on $c$ and $\sup_{v}|\psi'|(v)$.
\end{lem}

\begin{proof}
We first note that multiplying \eqref{eq:kinetic_ph-3} with a smooth approximation of $\sgn(v)$, integrating and taking the limit yields, for all $t\ge0$,
\begin{align*}
\int|u(t,x)|dx & \le\int|u(0,x)|+\|S\|_{L^{1}([0,T]\times\R^{d})}.
\end{align*}

From \eqref{eq:kinetic_ph-3} and a standard cut-off argument we further obtain
\begin{align*}
\partial_{t}\int\psi(u(t,x))dx & =\partial_{t}\int f(t,x,v)\psi'(v)dvdx\\
 & \le-\int\psi''(v)q(t,x,v)dvdx+\int S(t,x)\psi'(u(t,x))dx.
\end{align*}
Hence,
\begin{align*}
\int_{0}^{T}\int\psi''(v)q(t,x,v)dvdxdt & \le-\int\psi(u(\cdot,x))dx\Big|_{0}^{T}+\int_{0}^{T}\int S(r,x)\psi'(u(r,x))dxdr\\
 & \le c\int|u(0,x)|dx+c\int|u(T,x)|dx+C\|S\|_{L_{t,x}^{1}}\\
 & \le C(\|u_{0}\|_{L_{x}^{1}}+\|S\|_{L_{t,x}^{1}}).
\end{align*}
\end{proof}
We may now apply Lemma \ref{lem:av} to obtain
\begin{cor}
\label{cor:par-hyp-av}Let $u_{0}\in L^{1}(\R_{x}^{d})$, $S\in L^{1}([0,T]\times\R_{x}^{d})$, $a$, $b$ satisfy \eqref{eq:ph-as} and let $u$ be the entropy solution to \eqref{eq:par-hyp}. Further assume that the symbol $\mcL$ defined in \eqref{eq:ph-symbol} satisfies \eqref{eq:cdt1}, \eqref{eq:cdt2} for all $\gamma\in[0,1)$ large enough. Then, for all
\begin{align*}
s\in\big[0,\frac{\a}{\a+1}(\beta-\l)\big) & ,\quad p\in\Big[1,\frac{2\a+2}{2\a+1}\Big),
\end{align*}
all $\phi\in C_{c}^{\infty}(\R_{v})$, $\gamma\in[0,1)$ large enough and $\mcO\subset\subset\R^{d}$, there is a constant $C\ge0$ such that 
\begin{align}
\|\int f\phi\,dv\|_{L^{p}([0,T];W^{s,p}(\mcO))} & \le C(\|u_{0}\|_{L_{x}^{1}}+\|u_{0}\|_{L_{x}^{2-\gamma}}^{2-\gamma}+\|S\|_{L_{t,x}^{1}}+\|S\|_{L_{t,x}^{2-\gamma}}^{2-\gamma}+1).\label{eq:cor_est}
\end{align}
\end{cor}

\begin{proof}
We will derive \eqref{eq:cor_est} on the level of the approximating equation \eqref{eq:visc_approx}. By convergence of the approximating solutions $u^{\ve}$ and lower-semicontinuity of the norm this is sufficient. For notational simplicity we suppress the $\ve$-dependency in the following, but note that all estimates are uniform with respect to these parameters. As in \cite[Section 7]{CP03} we observe the bound (uniformly in $\ve$), for each $\psi\in C_{c}^{\infty}(\R_{v})$, $k=1,\dots,d$,
\[
\|\sum_{i=1}^{d}\partial_{x_{i}}\b_{ik}^{\psi}(u)\|_{L_{t,x}^{2}}\lesssim\|u_{0}\|_{L_{x}^{1}}+\|S\|_{L_{t,x}^{1}}+1.
\]
We hence estimate, for any $\vp\in C_{c}^{\infty}(\R_{t}\times\R_{x}^{d}\times\R_{v})$ and $\psi\in C_{c}^{\infty}(\R_{v})$ such that $\vp\psi=\vp$,
\begin{align}
\int_{t,x,v}|\nabla f\cdot b(v)\nabla\vp| & \le\sum_{k=1}^{d}\int_{t,x,v}\left|\left(\sum_{i=1}^{d}\partial_{x_{i}}f\s_{ik}(v)\right)\left(\sum_{j=1}^{d}\s_{kj}(v)\partial_{x_{j}}\vp\right)\right|\nonumber \\
 & =\sum_{k=1}^{d}\int_{t,x,v}\left|\left(\sum_{i=1}^{d}\delta_{u(t,x)=v}\partial_{x_{i}}u\s_{ik}(v)\psi(v)\right)\left(\sum_{j=1}^{d}\s_{kj}(v)\partial_{x_{j}}\vp\right)\right|\nonumber \\
 & =\sum_{k=1}^{d}\int_{t,x}\left|\left(\sum_{i=1}^{d}\partial_{x_{i}}\b_{ik}^{\psi}(u)\right)\left(\sum_{j=1}^{d}\s_{kj}(u)(\partial_{x_{j}}\vp)(t,x,u(t,x))\right)\right|\label{eq:grad-part-est}\\
 & \le\sum_{k=1}^{d}\|\sum_{i=1}^{d}\partial_{x_{i}}\b_{ik}^{\psi}(u)\|_{L_{t,x}^{2}}\|\sum_{j=1}^{d}\s_{kj}(u)(\partial_{x_{j}}\vp)(t,x,u(t,x))\|_{L_{t,x}^{2}}\nonumber \\
 & \lesssim\|u_{0}\|_{L_{x}^{1}}+\|S\|_{L_{t,x}^{1}}+1.\nonumber 
\end{align}

We next note that due to \eqref{eq:par_hyp_char_fctn} we have $\partial_{v}f(t,x,v)=\delta_{v=0}-\delta_{u(t,x)=v}$ and thus $f\in L_{t,x}^{\infty}(\dot{BV_{v}})\subseteq L_{t,x;loc}^{1}(\dot{BV_{v}})$ with $\|f\|_{L_{t,x}^{\infty}(\dot{BV_{v}})}\le2$. Moreover, by \eqref{eq:lp-bound},
\begin{equation}
\|f\|_{L_{t,x,v}^{1}}\lesssim\|u_{0}\|_{L_{x}^{1}}+\|S\|_{L_{t,x}^{1}}\label{eq:l1-bound}
\end{equation}
and $|f|\le1$. Hence, $f\in L_{t,x,v}^{1}\cap L_{t,x,v}^{\infty}$ and, for all $\s\in[0,\frac{1}{2})$,
\begin{align*}
\|f\|_{L_{t,x;loc}^{2}(B_{2,\infty}^{\s}(\R_{v}))}^{2} & =\|f\|_{L_{t,x;loc}^{2}(L_{v}^{2})}^{2}+\|\sup_{\delta>0}\sup_{0<|h|<\delta}\int_{\R}\frac{|f(t,x,v+h)-f(t,x,v)|^{2}}{|h|^{2\s}}dv\|_{L_{t,x;loc}^{2}}^{2}\\
 & \lesssim\|f\|_{L_{t,x;loc}^{1}(L_{v}^{1})}+\|\sup_{\delta>0}\sup_{0<|h|<\delta}\int_{\R}\frac{|f(t,x,v+h)-f(t,x,v)|}{|h|^{2\s}}dv\|_{L_{t,x;loc}^{2}}^{2}\\
 & \lesssim\|u_{0}\|_{L_{x}^{1}}+\|S\|_{L_{t,x}^{1}}+\|\|f(t,x,\cdot)\|_{\dot{BV_{v}}}\|_{L_{t,x;loc}^{2}}^{2}\\
 & \lesssim\|u_{0}\|_{L_{x}^{1}}+\|S\|_{L_{t,x}^{1}}+1,
\end{align*}
which implies, for all $\s\in[0,\frac{1}{2})$,
\begin{equation}
\|f\|_{L_{t,x;loc}^{2}(H_{v}^{\s,2})}\lesssim1+\|u_{0}\|_{L_{x}^{1}}+\|S\|_{L_{t,x}^{1}}.\label{eq:hs-bound}
\end{equation}
In order to apply Lemma \ref{lem:av} we hence have to localize $f$. Let $\vp\in C_{c}^{\infty}((0,T)\times\R_{x}^{d}\times\R_{v})$, $\eta^{\d}\in C^{\infty}(\R)$ satisfy $\eta^{\d}(v)\in[0,1]$ for all $v\in\R$, $|(\eta^{\d})'|\lesssim\frac{1}{\d}$, 
\begin{equation}
\eta^{\d}(v)=\begin{cases}
1 & \text{for }|v|\ge\d\\
0 & \text{for }|v|\le\frac{\d}{2}
\end{cases}\label{eq:cut_vel}
\end{equation}
and set $\vp^{\d}=\vp\eta^{\d}$. For simplicity we suppress the $\d$-index in the following. Set $\tilde{f}:=\vp f\in L_{t,x}^{2}(W_{v}^{\s,2})$, $\tilde{q}:=\vp q$. Then
\begin{align}
\partial_{t}\tilde{f} & =\vp\big(-a(v)\cdot\nabla f+\div(b(v)\nabla f)+\partial_{v}q+S\d_{u(t,x)=v}(v)\big)+f\partial_{t}\vp\nonumber \\
 & =(-a(v)\cdot\nabla\tilde{f}+\div(b(v)\nabla\tilde{f})+\partial_{v}\tilde{q}+\vp S\d_{u(t,x)=v}(v)\label{eq:localized}\\
 & +a(v)\cdot f\nabla\vp-2\nabla f\cdot b(v)\nabla\vp-f\div(b(v)\nabla\vp)-(\partial_{v}\vp)q+f\partial_{t}\vp.\nonumber 
\end{align}
Since $\vp$ is compactly supported and $q\in\mcM$, we have $\tilde{q}\in\mcM_{TV}$. Moreover, due to \eqref{eq:grad-part-est} and $S\in L_{t,x}^{1}$ we have 
\begin{align*}
g_{0}:= & \vp S\d_{u(t,x)=v}(v)+a(v)\cdot f\nabla\vp-2\nabla f\cdot b(v)\nabla\vp-f\div(b(v)\nabla\vp)\\
 & -(\partial_{v}\vp)q+f\partial_{t}\vp\in\mcM_{TV}
\end{align*}
with
\begin{equation}
\|g_{0}\|_{\mcM_{TV}}\le\|u_{0}\|_{L_{x}^{1}}+\|S\|_{L_{t,x}^{1}}+\|\partial_{v}\vp q\|_{\mcM_{TV}}+\|f\phi\partial_{t}\vp\|_{L_{t,x,v}^{1}}.\label{eq:g0}
\end{equation}
Let $s\in[0,\frac{\a}{\a+1}(\beta-\l))$ and $p\in[1,\frac{2\a+2}{2\a+1})$. Choose $\gamma\in[0,1)$ large enough and $r>1$ small enough, such that $s<(1-\theta)\frac{\a\b}{r}-\l\t$ and $p<\frac{2}{1+\t}$ where $\t=\frac{\frac{\a}{r}}{\frac{\a}{r}+1}$. We may assume $u_{0}\in L_{x}^{1}\cap L_{x}^{2-\gamma}$, $S\in L_{t,x}^{1}\cap L_{t,x}^{2-\gamma}$, otherwise there is nothing to be shown. By Lemma \ref{lem:ph-est} we have
\begin{align*}
 & \|u(t)\|_{L_{x}^{2-\gamma}}^{2-\gamma}+(1-\gamma)\int_{0}^{t}\int_{\R^{d+1}}|v|^{-\gamma}q\,dvdxdr\lesssim\|u_{0}\|_{L_{x}^{2-\gamma}}^{2-\gamma}+\|S\|_{L_{t,x}^{2-\gamma}}^{2-\gamma}.
\end{align*}

We note that, due to \eqref{eq:cut_vel} and \eqref{eq:ph-as} we may assume $a,b\in C^{1}$ without changing \eqref{eq:localized}. We now apply Lemma \ref{lem:av} with $\eta=0,$ $g_{1}=\td q$, $f=\td f$, $q=1$, $p=2$, $\s\in(0,\frac{1}{2})$ large enough, $T\ge0$, $\mcO\subseteq\R^{d}$ compact to obtain that there is a constant $C\ge0$ such that
\begin{align*}
\|\int f\vp^{\d}\phi\,dv\|_{L^{p}([0,T];\dot{W}^{s,p}(\mcO))} & \lesssim\|g_{0}^{\d}\phi\|_{\mcM_{TV}}+\||v|^{-\gamma}g_{1}^{\d}\phi\|_{\mcM_{TV}}+\|g_{1}^{\d}\phi'\|_{\mcM_{TV}}\\
 & +\|f\phi\|_{L_{t,x}^{p}(H_{v}^{\s,p})}+\|f\phi\|_{L_{t,x,v}^{1}}+\|f\phi\|_{L_{t}^{p}L_{x,v}^{1}}.
\end{align*}
Noting that 
\[
\|f\phi\|_{L_{t}^{p}L_{x,v}^{1}}\lesssim\|u\|_{L_{t}^{p}L_{x}^{1}}\lesssim\|u_{0}\|_{L_{x}^{1}},
\]
by Lemma \ref{lem:ph-est}, \eqref{eq:hs-bound}, \eqref{eq:l1-bound} and \eqref{eq:g0} we obtain that
\begin{align*}
\|\int f\vp^{\d}\phi\,dv\|_{L^{p}([0,T];\dot{W}^{s,p}(\mcO))} & \lesssim\|u_{0}\|_{L_{x}^{1}}+\|S\|_{L_{t,x}^{1}}+\|\partial_{v}\vp q\|_{\mcM_{TV}}+\|f\phi\partial_{t}\vp\|_{L_{t,x,v}^{1}}\\
 & +\|u_{0}\|_{L_{x}^{2-\gamma}}^{2-\gamma}+\|S\|_{L_{t,x}^{2-\gamma}}^{2-\gamma}+1.
\end{align*}
We next consider the limit $\d\to0$. Since $|\eta^{\d}|\le1$, the only nontrivial term appearing on the right hand side is $\|(\partial_{v}\eta^{\d})\vp q\|_{\mcM_{TV}}$. Let $\psi^{\d}$ be such that $(\psi^{\d})''=|\partial_{v}\eta^{\d}|$ and $|\psi^{\d}(r)|\le c|r|$. Then $\psi^{\d}$ satisfies the assumptions of Lemma \ref{lem:ph-est-1} uniformly in $\d$ which yields the required bound. Since $\vp$ is arbitrary, we conclude 
\begin{align*}
\|\int f\phi\,dv\|_{L^{p}([0,T];\dot{W}^{s,p}(\mcO))} & \lesssim\|u_{0}\|_{L_{x}^{1}}+\|u_{0}\|_{L_{x}^{2-\gamma}}^{2-\gamma}+\|S\|_{L_{t,x}^{1}}+\|S\|_{L_{t,x}^{2-\gamma}}^{2-\gamma}+1.
\end{align*}
Since $\phi$ is compactly supported, we have $\|\int f\phi\,dv\|_{L_{t,x}^{\infty}}\lesssim1$ which concludes the proof.
\end{proof}
\begin{thm}
\label{ex:anisoptropic_PME}Let $u_{0}\in L^{1}(\R_{x}^{d})$, $S\in L^{1}([0,T]\times\R_{x}^{d})$, $m_{j},n_{j}\ge1$, $j=1,\dots,d$ and let $u$ be the entropy solution to 
\begin{align}
\partial_{t}u+\sum_{j=1}^{d}\partial_{x_{j}}u^{n_{j}} & =\sum_{j=1}^{d}\partial_{x_{j}x_{j}}^{2}u^{[m_{j}]}+S(t,x)\quad\text{on }(0,T)\times\R^{d}\label{eq:par-hyp-2}\\
u(0) & =u_{0}\quad\text{on }\R^{d}.\nonumber 
\end{align}
We set $\underline{m}=\min(\{m_{j}:\,j=1,\dots,d\})$, $\overline{m}=\max(\{m_{j}:\,j=1,\dots,d\})$ and analogously $\underline{n}$, $\overline{n}$. Then, for all 
\[
s\in\big[1,\frac{2}{\overline{m}}\left(\frac{\underline{m}\wedge\underline{n}-1}{\overline{m}-1}\right)\big),\quad p\in\big[1,\frac{2\overline{m}}{1+\overline{m}}\big),
\]
all $\phi\in C_{c}^{\infty}(\R_{v})$, $\gamma\in[0,1)$ large enough and $\mcO\subset\subset\R^{d}$ there is a constant $C\ge0$ such that 
\begin{equation}
\|\int f\phi\,dv\|_{L^{p}([0,T];W^{s,p}(\mcO))}\le C\left(\|u_{0}\|_{L_{x}^{1}}+\|u_{0}\|_{L_{x}^{2-\gamma}}^{2-\gamma}+\|S\|_{L_{t,x}^{1}}+\|S\|_{L_{t,x}^{2-\gamma}}^{2-\gamma}+1\right).\label{eq:anis_reg_1}
\end{equation}
As a special case, for $m_{j}=n_{j}=m$, $j=1,\dots,d$, we obtain \eqref{eq:anis_reg_1} for all 
\[
s\in\big[0,\frac{2}{m}\big),\quad p\in\big[1,2\frac{m}{m+1}\big).
\]
\end{thm}

\begin{proof}
We have
\begin{align*}
\mcL(i\tau,i\xi,v) & =i\tau+i\sum_{j=1}^{d}n_{j}v^{n_{j}-1}\xi_{j}-\sum_{j=1}^{d}m_{j}|v|^{m_{j}-1}|\xi_{j}|^{2}\\
 & =:\mcL_{hyp}(i\tau,i\xi,v)+\mcL_{par}(\xi,v).
\end{align*}
Let $I\subseteq\R$ be a bounded set. Then, for $|\xi|\sim J$, 
\begin{align}
\Omega_{\mcL}(\tau,\xi;\d) & =\{v\in I:\,|\mcL(i\tau,i\xi,v)|\le\d\}\nonumber \\
 & \subseteq\Omega_{\mcL_{par}}(\xi;\d)=\{v\in I:\,\sum_{j=1}^{d}m_{j}|v|^{m_{j}-1}|\xi_{j}|^{2}\le\d\}\nonumber \\
 & \subseteq\{v\in I:\,|v|^{\overline{m}-1}J^{2}\lesssim\d\}.\label{eq:omega_est}
\end{align}
Thus,
\begin{align*}
|\Omega_{\mcL}(\tau,\xi;\d)| & \lesssim\left(\frac{\d}{J^{2}}\right)^{\frac{1}{\overline{m}-1}},
\end{align*}
i.e.~\eqref{eq:cdt1} is satisfied with $\b=2$, $\a=\frac{1}{\overline{m}-1}$. Moreover, due to \eqref{eq:omega_est}, for $|\xi|\sim J$, $v\in\Omega_{\mcL}(\tau,\xi;\d)\setminus\{0\}$,
\begin{align*}
|\partial_{v}\mcL(i\tau,i\xi,v)||v|^{\gamma} & =\Big|i\sum_{j=1}^{d}n_{j}(n_{j}-1)v^{n_{j}-2}\xi_{j}-\sum_{j=1}^{d}m_{j}(m_{j}-1)v{}^{[m_{j}-2]}|\xi_{j}|^{2}\Big||v|^{\gamma}\\
 & \lesssim|v|^{\underline{n}-2+\gamma}J+|v|^{\underline{m}-2+\gamma}J^{2}\\
 & \lesssim\d^{\frac{\underline{n}-2+\gamma}{\overline{m}-1}}J^{-\frac{2(\underline{n}-2+\gamma)}{\overline{m}-1}+1}+\d^{\frac{\underline{m}-2+\gamma}{\overline{m}-1}}J^{-\frac{2(\underline{m}-2+\gamma)}{\overline{m}-1}+2}.
\end{align*}
Using $\delta,J\ge1$ we get
\begin{align}
|\partial_{v}\mcL(i\tau,i\xi,v)||v|^{\gamma} & \lesssim\d^{\frac{\underline{m}\vee\underline{n}-2+\gamma}{\overline{m}-1}}J^{2-2\frac{\underline{m}\wedge\underline{n}-2+\gamma}{\overline{m}-1}},\label{eq:est_step}
\end{align}
i.e. \eqref{eq:cdt2} is satisfied with $\l=2-2\frac{\underline{m}\wedge\underline{n}-2+\gamma}{\overline{m}-1}$, $\mu=\frac{\underline{m}\vee\underline{n}-2+\gamma}{\overline{m}-1}$. An application of Corollary \ref{cor:par-hyp-av} with $\gamma$ close to one implies for all 
\begin{align*}
s & <s^{*}=\frac{2}{\overline{m}}\left(\frac{\underline{m}\wedge\underline{n}-1}{\overline{m}-1}\right),
\end{align*}
all $p<p^{*}=\frac{2\overline{m}}{1+\overline{m}}$ , all $\phi\in C_{c}^{\infty}(\R_{v})$, $\g\in[0,1)$ large enough, $\mcO\subset\subset\R^{d}$ that there is a constant $C\ge0$ and 
\begin{align*}
\|\int f\phi\,dv\|_{L^{p}([0,T];W^{s,p}(\mcO))} & \le C(\|u_{0}\|_{L_{x}^{1}}+\|u_{0}\|_{L_{x}^{2-\gamma}}^{2-\gamma}+\|S\|_{L_{t,x}^{1}}+\|S\|_{L_{t,x}^{2-\gamma}}^{2-\gamma}+1).
\end{align*}
\end{proof}
\begin{rem}
In Theorem \ref{ex:anisoptropic_PME} only the regularizing effect of the parabolic part is used. It may be possible that in cases $n_{j}<<m_{j}$ the hyperbolic regularizing effect would dominate. Since we are mostly interested in the parabolic regularization we do not consider this point here. For related work on hyperbolic averaging we refer to \cite{GL17}.  
\end{rem}

\section{Isotropic case\label{sec:Isotropic-case}}

In this section we consider parabolic-hyperbolic PDE with isotropic parabolic part, that is, 
\begin{align}
\partial_{t}f(t,x,v)+a(v)\cdot\nabla_{x}f(t,x,v)-b(v)\Delta_{x}f(t,x,v) & =:\mcL(\partial_{t},\nabla_{x},v)f(t,x,v)\label{eq:kinetic_isotropic}\\
 & =g_{0}(t,x,v)+\partial_{v}g_{1}(t,x,v),\nonumber 
\end{align}
where $a:\R\to\R^{d}$, $b:\R\to\R_{+}\cup\{0\}$ are twice continuously differentiable. The operator $\mcL$ is given by its symbol
\begin{align*}
\mcL(i\tau,i\xi,v) & :=\mcL_{hyp}(i\tau,i\xi,v)+\mcL_{par}(\xi,v)\\
 & :=i\tau+ia(v)\cdot\xi-b(v)|\xi|^{2},
\end{align*}
which by Appendix \ref{sec:Truncation-property-and} satisfies the truncation property uniformly in $v\in\R$. 

In this isotropic case we may work with a more restrictive non-degeneracy condition, which will allow to improve the order of integrability obtained in Theorem \ref{ex:anisoptropic_PME}.
\begin{defn}
[Isotropic truncation property]

\begin{enumerate}
\item We say that a function $m:\R_{\xi}^{d}\to\C$ is isotropic if $m$ is radial, that is, it depends only on $|\xi|^{2}$ .
\item Let $m:\R_{\xi}^{d}\times\R_{v}\to\C$ be a Caratheodory function such that $m(\cdot,v)$ is isotropic for all $v\in\R$. Then $m$ is said to satisfy the isotropic truncation property if for every bump function $\psi$ supported on a ball in $\C$, every bump function $\vp$ supported in $\{\xi\in\C:\,1\le|\xi|\le4\}$ and every $1<p<\infty$ 
\[
M_{\psi,J}f(x,v):=\mcF_{x}^{-1}\vp\left(\frac{|\xi|^{2}}{J^{2}}\right)\psi\left(\frac{m(\xi,v)}{\d}\right)\mcF_{x}f(x)
\]
is an $L_{x}^{p}$-multiplier for all $v\in\R$, $J=2^{j},\,j\in\N$ and, for all $r\ge1$,
\begin{align*}
\Big\|\|M_{\psi,J}\|_{\mcM^{p}}\Big\|_{L_{v}^{r}} & \lesssim|\Omega_{m}(J,\d)|^{\frac{1}{r}},
\end{align*}
where
\[
\Omega_{m}(J,\d):=\{v\in\R:\,|\frac{m(J,v)}{\d}|\in\supp\psi\}.
\]
\end{enumerate}
\end{defn}

\begin{example}
Consider 
\[
\mcL(\xi,v)=-|\xi|^{2}b(v),
\]
for $b:\R\to\R_{+}\cup\{0\}$ being measurable. Then $\mcL$ satisfies the isotropic truncation property.
\end{example}

\begin{proof}
Let $\vp$, $\psi$ be as in the definition of the isotropic truncation property. In order to prove that $M_{\psi,J}$ is an $L^{p}$-multiplier we will invoke the Hörmander\textendash Mihlin Multiplier Theorem \cite[Theorem 5.2.7]{G14-2}. We note that 
\[
\sup_{\xi\in\R^{d}}\vp\left(\frac{|\xi|^{2}}{J^{2}}\right)\psi\left(\frac{\mcL(\xi,v)}{\d}\right)<\infty
\]
and 
\begin{align*}
 & \partial_{\xi_{i}}\vp\left(\frac{|\xi|^{2}}{J^{2}}\right)\psi\left(\frac{\mcL(\xi,v)}{\d}\right)\\
 & =\vp'\left(\frac{|\xi|^{2}}{J^{2}}\right)\frac{|\xi|^{2}}{J^{2}}\frac{2\xi_{i}}{|\xi|^{2}}\psi\left(\frac{\mcL(\xi,v)}{\d}\right)+\vp\left(\frac{|\xi|^{2}}{J^{2}}\right)\psi'\left(\frac{\mcL(\xi,v)}{\d}\right)\frac{\mcL(\xi,v)}{\d}\frac{2\xi_{i}}{|\xi|^{2}}\\
 & =\left[\vp'\left(\frac{|\xi|^{2}}{J^{2}}\right)\frac{|\xi|^{2}}{J^{2}}\psi\left(\frac{\mcL(\xi,v)}{\d}\right)+\vp\left(\frac{|\xi|^{2}}{J^{2}}\right)\psi'\left(\frac{\mcL(\xi,v)}{\d}\right)\frac{\mcL(\xi,v)}{\d}\right]\frac{2\xi_{i}}{|\xi|^{2}}\\
 & =\tilde{\vp}\left(\frac{|\xi|^{2}}{J^{2}}\right)\tilde{\psi}\left(\frac{\mcL(\xi,v)}{\d}\right)\frac{2\xi_{i}}{|\xi|^{2}},
\end{align*}
where $\tilde{\vp}$, $\tilde{\psi}$ are bump functions with the same support properties as $\vp,\psi$. Hence, induction yields
\begin{align*}
|\partial_{\xi}^{\a}\vp\left(\frac{|\xi|^{2}}{J^{2}}\right)\psi\left(\frac{\mcL(\xi,v)}{\d}\right)| & \le\tilde{\vp^{\a}}\left(\frac{|\xi|^{2}}{J^{2}}\right)\tilde{\psi^{\a}}\left(\frac{\mcL(\xi,v)}{\d}\right)\frac{C_{\a}}{|\xi|^{|\a|}},
\end{align*}
for all multi-indices $\a$ with $|\a|\le[\frac{d}{2}]+1,$ where where $\tilde{\vp}^{\a}$, $\tilde{\psi^{\a}}$ are bump functions with the same support properties as $\vp,\psi$. The Hörmander\textendash Mihlin Multiplier Theorem thus implies that 
\[
\vp\left(\frac{|\xi|^{2}}{J^{2}}\right)\psi\left(\frac{\mcL(\xi,v)}{\d}\right)\in\mcM^{p}
\]
for all $1<p<\infty$ with 
\begin{align*}
\|\vp\left(\frac{|\xi|^{2}}{J^{2}}\right)\psi\left(\frac{\mcL(\xi,v)}{\d}\right)\|_{\mcM^{p}} & \le C_{d,p}\sup_{\xi\in\R^{d}}\tilde{\vp}\left(\frac{|\xi|^{2}}{J^{2}}\right)\tilde{\psi}\left(\frac{\mcL(\xi,v)}{\d}\right),
\end{align*}
where $\tilde{\vp}$, $\tilde{\psi}$ are bump functions as above. Hence, 
\begin{align*}
\|\vp\left(\frac{|\xi|^{2}}{J^{2}}\right)\psi\left(\frac{\mcL(\xi,v)}{\d}\right)\|_{\mcM^{p}} & \le C_{d,p}\sup_{J\le|\xi|\le2J}\tilde{\psi}\left(\frac{\mcL(\xi,v)}{\d}\right).
\end{align*}

Hence,
\begin{align*}
 & \Big\|\|\vp\left(\frac{|\xi|^{2}}{J^{2}}\right)\psi\left(\frac{\mcL(\xi,v)}{\d}\right)\|_{\mcM^{p}}\Big\|_{L_{v}^{r}}\lesssim\left(\int\sup_{J\le|\xi|\le2J}\tilde{\psi}\left(\frac{\mcL(\xi,v)}{\d}\right)\,dv\right)^{\frac{1}{r}}\\
 & \lesssim\left(\int\sup_{J\le|\xi|\le2J}1_{\frac{|\xi|^{2}b(v)}{\d}\in\supp\tilde{\psi}}\,dv\right)^{\frac{1}{r}}\lesssim\left(\int1_{\frac{|J|^{2}b(v)}{\d}\in\supp\tilde{\psi}}\,dv\right)^{\frac{1}{r}}\\
 & \lesssim\left(|\{v\in\R:\,\frac{|J|^{2}b(v)}{\d}\in\supp\tilde{\psi}|\}\right)^{\frac{1}{r}}=|\Omega_{\mcL}(J,\d)|^{\frac{1}{r}}.
\end{align*}

\end{proof}

\subsection{Averaging Lemma }

Working with the isotropic truncation property allows to prove a similar statement to Lemma \ref{lem:av}, but without the restriction to $p\le2$. This leads to an improved estimate on the integrability of the solution. 
\begin{lem}
\label{lem:av-iso}

Let $f\in L_{v}^{r'}(L_{t,x}^{p})$ for $1<p<\infty$, $r'\in(1,\infty]$ solve, in the sense of distributions,
\begin{equation}
\mcL(\partial_{t},\nabla_{x},v)f(t,x,v)=\D_{x}^{\frac{\eta}{2}}g_{0}(t,x,v)+\partial_{v}\D_{x}^{\frac{\eta_{}}{2}}g_{1}(t,x,v)\text{ on }\R_{t}\times\R_{x}^{d}\times\R_{v}\label{eq:eqn-1}
\end{equation}
with $g_{i}$ being Radon measures satisfying 
\begin{equation}
|g_{0}|(t,x,v)+|g_{1}|(t,x,v)|v|^{-\gamma}\in\begin{cases}
L^{q}(\R_{t}\times\R_{x}^{d}\times\R_{v}), & 1<q\le2\\
\mcM_{TV}(\R_{t}\times\R_{x}^{d}\times\R_{v}), & q=1,
\end{cases}\label{eq:measure_bound-1}
\end{equation}
for some $\gamma\ge0$, $\eta\ge0$, $1\le q\le\min(p,2)$ and $\mcL(\partial_{t},\nabla_{x},v)$ as in \eqref{eq:kinetic_isotropic} with corresponding symbol $\mcL(i\tau,i\xi,v)=\mcL_{hyp}(i\tau,i\xi,v)+\mcL_{par}(\xi,v)$. Let $I\subseteq\R$ be a not necessarily bounded interval, set 
\[
\o_{\mcL}(J;\d):=\sup_{\tau\in\R,\,\xi\in\R^{d},|\xi|\sim J}|\Omega_{\mcL}(\tau,\xi;\d)|,\quad\Omega_{\mcL}(\tau,\xi;\d)=\{v\in I:\,|\mcL(i\tau,i\xi,v)|\le\d\},
\]
and suppose that the following non-degeneracy condition holds: There exist $\a,\b>0$ such that 
\begin{equation}
\o_{\mcL}(J;\d)\lesssim(\frac{\d}{J^{\b}})^{\a}\quad\forall\d\ge1,\ J\ge1.\label{eq:cdt1-iso}
\end{equation}
Moreover, assume that there exist $\l\ge0$ and $\mu\in[0,1]$ such that, for all $\d\ge1,$ $J\ge1$, 
\begin{equation}
\sup_{\tau,|\xi|\sim J}\sup_{v\in\Omega_{\mcL}(\tau,\xi;\d)}|\partial_{v}\mcL(i\tau,i\xi,v)||v|^{\gamma}\lesssim J^{\l}\d^{\mu}\label{eq:cdt2-iso}
\end{equation}
and $\frac{\a\b}{q'}\le\l+\eta$. Assume that $\mcL_{par}$ satisfies the isotropic truncation property with
\begin{equation}
|\Omega_{\mcL_{par}}(J,\d)|\lesssim(\frac{\d}{J^{\b}})^{\a}\quad\forall\d\ge1,\ J\ge1.\label{eq:cdt3-iso}
\end{equation}
Then, for all $\phi\in C_{b}^{\infty}(I)$, $s\in[0,s^{*})$, $\td p\in[1,p^{*})$, $T\ge0$, $\mcO\subset\subset\R^{d}$, there is a constant $C\ge0$ such that
\begin{align}
\|\int f(t,x,v)\phi(v) & \,dv\|_{L^{\td p}([0,T];\dot{W}^{s,\td p}(\mcO))}\le C\big(\|g_{0}\phi\|_{L_{t,x,v}^{q}}+\||v|^{-\gamma}g_{1}\phi\|_{L_{t,x,v}^{q}}\label{eq:iso_ineq}\\
 & +\|g_{1}\phi'\|_{L_{t,x,v}^{q}}+\|f\phi\|_{L_{v}^{r'}(L_{t,x}^{p})}+\|f\phi\|_{L_{t,x}^{q}L_{v}^{1}}+\|f\phi\|_{L_{t}^{\td p}L_{x,v}^{1}}\big)\nonumber 
\end{align}
with $s^{*}:=(1-\t)\frac{\a\b}{r}+\t(\frac{\a\b}{q'}-\l-\eta),$ where $\t=\theta_{\a}$ and $p^{*}$ are given by
\[
\t:=\frac{\frac{\a}{r}}{\a(\frac{1}{r}-\frac{1}{q'})+1}\in(0,1),\quad\frac{1}{p^{*}}:=\frac{1-\t}{p}+\frac{\t}{q},\ \frac{1}{r}+\frac{1}{r'}=1.
\]
An analogous estimate can be given for inhomogeneous Sobolev spaces.
\end{lem}

\begin{proof}
The proof proceeds analogously to the one of Lemma \ref{lem:av}. The only change appears in the estimation of $f^{0}.$ We may assume that $\psi_{0}$ is of the form $\psi_{0}(ia+b)=\psi_{0}^{1}(a)\psi_{0}^{2}(b)$ with $\psi_{0}^{i}$ being locally supported bump functions. Hence,
\[
\psi_{0}\left(\frac{\mcL(i\tau,i\xi,v)}{\d}\right)=\psi_{0}^{1}\left(\frac{\mcL_{hyp}(i\tau,i\xi,v)}{\d}\right)\psi_{0}^{2}\left(\frac{\mcL_{par}(\xi,v)}{\d}\right)
\]
and 
\[
\|\vp_{1}(\frac{\xi}{2^{j}})\psi_{0}\left(\frac{\mcL(i\tau,i\xi,v)}{\d}\right)\|_{\mcM^{p}}\lesssim\|\vp_{1}(\frac{\xi}{2^{j}})\psi_{0}^{2}\left(\frac{\mcL_{par}(\xi,v)}{\d}\right)\|_{\mcM^{p}}.
\]
The isotropic truncation property and \eqref{eq:cdt3-iso} then imply
\[
\Big\|\|\vp_{1}(\frac{\xi}{2^{j}})\psi_{0}\left(\frac{\mcL(i\tau,i\xi,v)}{\d}\right)\|_{\mcM^{p}}\Big\|_{L_{v}^{r}}\lesssim|\Omega_{\mcL_{par}}(2^{j},\d)|^{\frac{1}{r}}\lesssim(\frac{\d}{2^{j\b}})^{\frac{\a}{r}}.
\]

Hence,
\begin{align*}
\|\int f_{j}^{0}\phi\,dv\|_{L_{t,x}^{p}} & =\|\int\mcF_{t,x}^{-1}\vp_{1}(\frac{\xi}{2^{j}})\psi_{0}\left(\frac{\mcL(i\tau,i\xi,v)}{\d}\right)\mcF_{t,x}f^{0}\phi\,dv\|_{L_{t,x}^{p}}\\
 & \le\int\|\mcF_{t,x}^{-1}\vp_{1}(\frac{\xi}{2^{j}})\psi_{0}\left(\frac{\mcL(i\tau,i\xi,v)}{\d}\right)\mcF_{t,x}f^{0}\phi\|_{L_{t,x}^{p}}\,dv\\
 & \lesssim\int\|\mcF_{t,x}^{-1}\vp_{1}(\frac{\xi}{2^{j}})\psi_{0}\left(\frac{\mcL(i\tau,i\xi,v)}{\d}\right)\mcF_{t,x}\|_{\mcM^{p}}\|f^{0}\phi\|_{L_{t,x}^{p}}\,dv\\
 & \le\Big\|\|\vp_{1}(\frac{\xi}{2^{j}})\psi_{0}\left(\frac{\mcL(i\tau,i\xi,v)}{\d}\right)\|_{\mcM^{p}}\Big\|_{L_{v}^{r}}\|f^{0}\phi\|_{L_{v}^{r'}L_{t,x}^{p}}\\
 & \lesssim(\frac{\d}{2^{j\b}})^{\frac{\a}{r}}\|f^{0}\phi\|_{L_{v}^{r'}L_{t,x}^{p}}.
\end{align*}
The proof then proceeds as before, the only difference being that we do not have to restrict to $1<p\le2$ and the modified definition of $r,r'$.
\end{proof}

\subsection{Porous media equations}

In this section we consider porous media equations with a source of the type
\begin{align}
\partial_{t}u & =\D u^{[m]}+S(t,x)\text{ on }(0,T)\times\R_{x}^{d},\label{eq:PME-inhomo}\\
u(0) & =u_{0},\nonumber 
\end{align}
where $u_{0}\in L^{1}(\R_{x}^{d}),\ S\in L^{1}([0,T]\times\R_{x}^{d})$, $T\ge0$ and $m>1$. 

As in \cite{CP03}, the kinetic form to \eqref{eq:PME-inhomo} reads, with $f=\chi(u(t,x),v),$ $q\in\mcM^{+}$,
\begin{align}
\partial_{t}f & =m|v|^{m-1}\Delta f+\partial_{v}q+S(t,x)\d_{u(t,x)}(v)\text{ on }(0,T)\times\R_{x}^{d}\times\R_{v}.\label{eq:kinetic_ph-2-1}
\end{align}
For the notion and well-posedness of entropy solutions to \eqref{eq:PME-inhomo} see Appendix \ref{app:kin_solutions}. As before, let $\mcL(\partial_{t},\nabla_{x},v)f=\partial_{t}f-m|v|^{m-1}\Delta f$ with symbol 
\begin{align*}
\mcL(i\tau,\xi,v):= & \mcL_{hyp}(i\tau)+\mcL_{par}(\xi,v)\\
:= & i\tau-m|v|^{m-1}|\xi|^{2}.
\end{align*}

\begin{thm}
\label{thm:pme}Let $u_{0}\in(L^{1}\cap L^{1+\ve})(\R_{x}^{d})$, $S\in(L^{1}\cap L^{1+\ve})([0,T]\times\R_{x}^{d})$ for some $\ve>0$. Let $u$ be the unique entropy solution to \eqref{eq:PME-inhomo}. Then, for all
\[
s\in[0,\frac{2}{m}),\quad p\in[1,m)
\]
  we have
\[
u\in L^{p}([0,T];\dot{W}_{loc}^{s,p}(\R_{x}^{d})).
\]
In addition, for all $\mcO\subset\subset\R^{d}$ there is a constant $C=C(m,p,s,\ve,T,\mcO)$ such that
\[
\|u\|_{L^{p}([0,T];\dot{W}^{s,p}(\mcO))}\le C\left(\|u_{0}\|_{L_{x}^{1}\cap L_{x}^{1+\ve}}^{2}+\|S\|_{L_{t,x}^{1}\cap L_{t,x}^{1+\ve}}^{2}+1\right).
\]
\end{thm}

\begin{proof}
Let $s\in[0,\frac{2}{m}),\,p\in[1,m)$. We have $f\in L_{t,x,v}^{1}\cap L_{t,x,v}^{\infty}$ and thus $f\in L_{v}^{\td p}(L_{t,x}^{\td p})$ for all $\td p\ge1$ with 
\begin{equation}
\|f\|_{L_{v}^{\td p}(L_{t,x}^{\td p})}^{\td p}\le\|f\|_{L_{v}^{1}(L_{t,x}^{1})}.\label{eq:f_lp_bound}
\end{equation}
This bound will replace the property $f\in L_{t,x;loc}^{2}(H_{v}^{\s,2})$ used in the proof of Corollary \ref{cor:par-hyp-av}, which is possible due to Lemma \ref{lem:av-iso}. As a consequence, the localization of $f$ performed in Corollary \ref{cor:par-hyp-av} is not required here. In order to apply \eqref{lem:av-iso} we need to extend \eqref{eq:kinetic_ph-2-1} to all time $t\in\R$, which can be done by multiplication with a smooth cut-off function $\vp\in C_{c}^{\infty}(0,T)$. Let $\eta=0$, $\a=\frac{1}{m-1}$, $\b=2$ and choose $\gamma\in[0,1)$ large enough and $r\ge1$ small enough such that $\l=2-2\frac{m-2+\gamma}{m-1}=2(\frac{1-\gamma}{m-1})$ is such that 
\begin{align*}
(1-\t)\b\frac{\a}{r}-\t(\l+\eta) & =\t(\frac{\b}{r}-\l)\\
 & =\frac{2}{m}\left(\frac{1}{r}-(\frac{1-\gamma}{m-1})\right)\\
 & >s,
\end{align*}
where $\t=\frac{1}{m}$. Next, choose $\td p$ large enough, such that $p^{*}=m\left(\frac{\td p}{m-1+\td p}\right)>p$ and note $\frac{1-\t}{\td p}+\t=\frac{1}{p^{*}}$. We can choose $\td p$, $r$ such that $\td p=r'$. Let $g_{0}=\delta_{v=u(t,x)}S+f\partial_{t}\vp$, $g_{1}=q$. In order to treat the possible singularity of $\partial_{v}\mcL$ at $v=0$ we proceed as in Corollary \ref{cor:par-hyp-av}, i.e.~first cutting out the singularity, then controlling the respective error uniformly by Lemma \ref{lem:ph-est-1}. Note that $\mcL$ satisfies \eqref{eq:cdt1-iso}, \eqref{eq:cdt2-iso} on $\R\setminus{\{0\}}$ for all $\gamma\in[0,1)$ and $\mcL_{par}$ satisfies the isotropic truncation property with \eqref{eq:cdt3-iso}. With these choices, Lemma \ref{lem:av-iso} with $p=\td p$, $q=1$ and $\phi\equiv1$ yields 
\begin{align*}
\|u\|_{L^{p}([0,T];\dot{W}^{s,p}(\mcO))} & \lesssim\|\delta_{v=u(t,x)}S\|_{\mcM_{TV}}+\|f_{0}\|_{L_{x}^{1}L_{v}^{1}}+\||v|^{-\gamma}q\|_{\mcM_{TV}}\\
 & +\|f\|_{L_{v}^{\td p}(L_{t,x}^{\td p})}+\|f\|_{L_{t,x}^{1}L_{v}^{1}}+\|f\|_{L_{t,x}^{p}L_{v}^{1}}\\
 & \lesssim\|S\|_{L_{t,x}^{1}}+\|u_{0}\|_{L_{x}^{1}}+\||v|^{-\gamma}q\|_{\mcM_{TV}}+\|f\|_{L_{t,x}^{1}L_{v}^{1}}+\|f\|_{L_{t}^{p}L_{x,v}^{1}}+1.
\end{align*}
The fact that, for all $\eta\in[1,\infty)$,
\begin{align*}
\|f\|_{L_{t,x}^{\eta}L_{v}^{1}} & =\|u\|_{L_{t,x}^{\eta}}\lesssim\|u_{0}\|_{L_{x}^{\eta}}+\|S\|_{L_{t,x}^{\eta}},\\
\|f\|_{L_{t}^{\eta}L_{x,v}^{1}} & =\|u\|_{L_{t}^{\eta}L_{x}^{1}}\lesssim\|u_{0}\|_{L_{x}^{1}}+\|S\|_{L_{t,x}^{1}}
\end{align*}
and Lemma \ref{lem:ph-est} thus imply
\begin{align*}
\|u\|_{L^{p}([0,T];\dot{W}^{s,p}(\mcO))} & \lesssim\|u_{0}\|_{L_{x}^{1}\cap L_{x}^{2-\g}}^{2}+\|S\|_{L_{t,x}^{1}\cap L_{t,x}^{2-\g}}^{2}+1.
\end{align*}
Since $p^{*}>p$, choosing $\g\in(0,1)$ large enough so that $2-\g\le1+\ve$ yields the claim.
\end{proof}

\begin{rem}
We note that for $u_{0}\in L_{x}^{1}$ or $S\in L_{t,x}^{1}$ the kinetic measure $q$ does not necessarily have finite mass (cf.~e.g.~\cite{P02}). Therefore, in the literature the cut-off $\phi\in C_{c}^{\infty}(\R)$ in \eqref{eq:iso_ineq} is required to be compactly supported, which prevents to deduce regularity estimates for $u$ itself, unless $u$ is bounded. Our arguments allow to avoid this restriction since we work with the singular moments $|v|^{-\g}q$ only, which are shown to be finite in Lemma \ref{lem:ph-est}, provided $u\in L_{x}^{2-\g}$, $S\in L_{t,x}^{2-\g}$.
\end{rem}

\begin{rem}
\label{rem:eb-technique}

As it has been pointed out in the introduction, the results obtained in \cite{E05} are restricted to fractional differentiability of an order less than one. This restriction is inherent to the method used in \cite{E05}. More precisely, the estimates obtained in \cite{E05} are (informally) based on testing \eqref{eq:PME-inhomo} with $\int_{0}^{t}\D u^{[m]}\,dr$, integrating in space and time and using Hölder's inequality, which leads to the energy inequality (neglecting constants) 
\begin{align}
 & \int_{0}^{T}\int(\nabla u{}^{[\frac{m+1}{2}]})^{2}dxdr\le\int u^{2}(0)dx.\label{eq:ebmeyer_energy_ineq}
\end{align}
The regularity estimates are then deduced from \eqref{eq:ebmeyer_energy_ineq} alone. In \cite{E05} these formal computations are made rigorous, a careful treatment of boundary conditions is given and the bound on $\int_{0}^{T}\int(\nabla u{}^{[\frac{m+1}{2}]})^{2}dxdr$ is used to prove \eqref{eq:eb-reg-1}. Since \eqref{eq:ebmeyer_energy_ineq} only involves derivatives of first order, it does not seem possible to deduce higher than first order differentiability from this.
\end{rem}

\section{Degenerate parabolic Anderson model\label{sec:Degenerate-parabolic-Anderson}}

We consider the degenerate parabolic Anderson model
\begin{align}
\partial_{t}u & =\partial_{xx}u^{[m]}+u\,S\text{ on }(0,T)\times I,\label{eq:anderson}\\
u^{\ve} & =0\text{ on }\partial I,\nonumber 
\end{align}
with $m\in(1,2)$, $I\subseteq\R$ a bounded, open interval and $S$ being a distribution only. As for the parabolic Anderson model (cf.~\cite{GM98,GM90}), the particular example we have in mind is $S=\xi$ being spatial white noise. Accordingly, we assume that, locally on $\R$, 
\begin{equation}
S\in B_{\infty,\infty}^{-\frac{1}{2}-\ve}\text{ for all }\ve>0.\label{eq:distr_ass}
\end{equation}
The choice of zero Dirichlet boundary data in \eqref{eq:anderson} is for simplicity only and the arguments of this section can easily be adapted to the Cauchy problem. 

We define weak solutions to \eqref{eq:anderson} to be functions $u\in L^{2}([0,T];H_{0}^{1}(I))$ such that $u^{[m]}\in L^{2}([0,T];H_{0}^{1}(I))$ and \eqref{eq:anderson} is satisfied in the sense of distributions. We will prove the following regularity estimate for a weak solution to \eqref{eq:anderson}.
\begin{cor}
\label{cor:reg_weak}Let $u_{0}\in L^{m+1}(I).$ Then there exists a weak solution $u$ to \eqref{eq:anderson} satisfying, for all $p\in[1,m)$, $s\in[0,\frac{3}{2}\frac{1}{m})$, 
\[
u\in L{}^{p}([0,T];W_{loc}^{s,p}(I)),
\]
with, for all $T\ge0$, $\mcO\subset\subset I$, 
\begin{align*}
\|u\|_{L^{p}([0,T];W^{s,p}(\mcO))} & \lesssim\|u_{0}\|_{L^{m+1}(I)}^{m+1}+\|S\|_{B_{\infty,\infty}^{-\eta}}^{\tau}+1,
\end{align*}
for some $\tau\ge2$ and $\eta\in(\frac{1}{2},1]$ small enough.
\end{cor}

The proof of the above Proposition is a consequence of establishing according uniform regularity estimates (see Theorem \ref{thm:approx_anderson} below) for the approximating problem
\begin{align}
\partial_{t}u^{\ve} & =\partial_{xx}(u^{\ve})^{[m]}+u^{\ve}S^{\ve}(x)\text{ on }(0,T)\times I,\label{eq:deg_Anderson_approx}\\
u^{\ve} & =0\text{ on }\partial I,\nonumber 
\end{align}
where $S^{\ve}\in C^{\infty}(\R)$ with $\|S^{\ve}\|_{B_{\infty,\infty}^{-\frac{1}{2}-\ve}}\le\|S\|_{B_{\infty,\infty}^{-\frac{1}{2}-\ve}}$ and $S^{\ve}\to S$ locally in $B_{\infty,\infty}^{-\frac{1}{2}-\ve}$ for all $\ve>0$. These estimates will be derived from the kinetic formulation of \eqref{eq:deg_Anderson_approx}. Informally, with $\chi^{\ve}:=\chi(u^{\ve})$ the kinetic form reads, in the sense of distributions,
\begin{align}
\partial_{t}\chi^{\ve} & =m|v|^{m-1}\partial_{xx}\chi^{\ve}+\d_{u^{\ve}(t,x)=v}u^{\ve}S^{\ve}+\partial_{v}q^{\ve}\nonumber \\
 & =m|v|^{m-1}\partial_{xx}\chi^{\ve}+\chi^{\ve}S^{\ve}+\partial_{v}q^{\ve}-\partial_{v}(\chi^{\ve}vS^{\ve})\text{ on }(0,T)\times I\times\R.\label{eq:kinetic_anderson}
\end{align}

\begin{defn}
\label{def:kinetic_sol}We say that $u^{\ve}\in L^{1}([0,T]\times I)$ is an entropy solution to \eqref{eq:deg_Anderson_approx} if 

\begin{enumerate}
\item [(i)] for every $\a\in(0,m]$ there is a constant $K_{1}\ge0$ such that
\begin{equation}
\|\partial_{x}(u^{\ve}){}^{[\frac{m+\a}{2}]}\|_{L^{2}([0,T]\times I)}\le K_{1}.\label{eq:kinetic_1}
\end{equation}
\item [(ii)] $\chi^{\ve}=\chi(u^{\ve})$ satisfies \eqref{eq:kinetic_anderson}, in the sense of distributions on $(0,T)\times I\times\R$, for some non-negative, finite measure $q^{\ve}$ such that,
\[
q^{\ve}=m^{\ve}+n^{\ve}
\]
with $m^{\ve}$ being a non-negative measure and $n^{\ve}$ given by
\[
n^{\ve}=\d_{v=u^{\ve}}(\partial_{x}(u^{\ve}){}^{[\frac{m+1}{2}]})^{2}
\]
and satisfying, for every $\a\in(0,m]$ with $K_{1}$ as in (i),
\begin{equation}
\int_{[0,T]\times\R^{d}\times\R}|v|^{\a-1}q^{\ve}\,dtdxdv\le K_{1}.\label{eq:kinetic_bound}
\end{equation}
 
\end{enumerate}
\end{defn}

The well-posedness of entropy solutions to \eqref{eq:deg_Anderson_approx} follows along the lines of Theorem \ref{thm:wp-kinetic} in Appendix \ref{app:kin_solutions} below. It only remains to show that the constant $K_{1}$ in \eqref{eq:kinetic_1} and \eqref{eq:kinetic_bound} can be chosen uniformly in $\ve$.
\begin{lem}
\label{lem:anders_bound}Let $\a>0$, $\tau=\frac{2\a+2}{2\a+3-m}\in(1,2]$ and $u_{0}\in(L^{m+1}\cap L^{\a+1})(\R_{x}^{d})$. Then, for some constant $C=C(\a,m,T)$, 
\begin{align*}
\sup_{t\in[0,T]}\int_{I}|u^{\ve}(t)|^{\a+1}dx & +\int_{0}^{T}\int_{I}(\partial_{x}(u^{\ve}){}^{[\frac{m+\a}{2}]})^{2}dxdr\le C\int_{I}|u_{0}|^{\a+1}dx+C\|S\|_{W^{-1,\tau'}}^{\tau'}.
\end{align*}
and 
\begin{align}
 & \int_{[0,T]\times\bar{I}\times\R}|v|^{\a-1}q^{\ve}\,drdxdv\le C\int_{I}|u_{0}|^{\a+1}dx+C\|S\|_{W^{-1,\tau'}}^{\tau'}.\label{eq:kin_meas_est}
\end{align}
\end{lem}

\begin{proof}
First, let $u_{0}\in C_{c}^{\infty}(\R_{x}^{d})$, $b^{\d}\in C^{\infty}(\R)$ increasing with $b^{\d}(u)\ge\d u$ for all $u\in\R$, $b^{\d}(u)\to u^{[m]}$ locally uniformly and let $u^{\ve,\d}$ be the classical solution to the approximating equation
\[
\partial_{t}u^{\ve,\d}=\partial_{xx}b^{\d}(u^{\ve,\d})+u^{\ve,\d}S^{\ve}(x)\text{ on }(0,T)\times I.
\]
For simplicity we drop the $\ve$ in the notation. Then, for $\eta\in C^{2}(\R)$ convex, Lipschitz continuous, we obtain
\begin{align*}
\int_{I}\eta(u^{\d}(t))dx & =\int_{I}\eta(u_{0})dx+\int_{0}^{t}\int_{I}\eta'(u^{\d})(\partial_{xx}b^{\d}(u^{\d})+u^{\d}S)\,dxdr\\
 & \le\int_{I}\eta(u_{0})dx-c\int_{0}^{t}\int_{I}(\partial_{x}F^{\eta}(u^{\d}))^{2}dxdr+\int_{0}^{t}\int_{I}\eta'(u^{\d})u^{\d}S\,dxdr,
\end{align*}
$F^{\eta}(u):=\int_{0}^{u}\sqrt{\eta''(r)(b^{\d})'(r)}dr.$ Hence, (for a non-relabeled subsequence) we have $\partial_{x}F^{\eta}(u^{\d})\rightharpoonup Z$ for some $Z\in L^{2}([0,T];L^{2}(\R_{x}^{d}))$. Since $u^{\d}\to u$ in $C([0,T];L^{1}(\R_{x}^{d}))$ we have $Z=\partial_{x}F^{\eta}(u)$ which implies 
\begin{align*}
\int_{I}\eta(u(t))dx & \le\int_{I}\eta(u_{0})dx-c\int_{0}^{t}\int_{I}(\partial_{x}F^{\eta}(u))^{2}dxdr+\int_{0}^{t}\int_{I}\eta'(u)uS\,dxdr,
\end{align*}
where $F^{\eta}(u):=m\int_{0}^{u}\sqrt{\eta''(r)|r|^{m-1}}dr$.

Using a suitable approximation of $\eta(u)=|u|^{\a+1}$ this yields, for some $c=c(\a,m)$,
\begin{align*}
\int_{I}|u(t)|^{\a+1}dx & \lesssim\int_{I}|u_{0}|^{\a+1}dx-c\int_{0}^{t}\int_{I}(\partial_{x}u{}^{[\frac{m+\a}{2}]})^{2}dxdr+\int_{0}^{t}\int_{I}|u|^{\a+1}Sdxdr.
\end{align*}
We further have, for $\tau\in[1,2)$ to be chosen later, 
\begin{equation}
\int_{I}|u|^{\a+1}Sdx\lesssim\||u|^{\a+1}\|_{W^{1,\tau}}^{\tau}+\|S\|_{W^{-1,\tau'}}^{\tau'}\label{eq:prod_est}
\end{equation}
and, for every $\eta>0$ and some $C_{\eta}\ge0$,
\begin{align}
\||u|^{\a+1}\|_{W^{1,\tau}}^{\tau} & \lesssim\int_{I}|\partial_{x}|u|^{\a+1}|^{\tau}\,dx=(\a+1)^{\tau}\int_{I}|u^{[\a]}\partial_{x}u|^{\tau}\,dx\nonumber \\
 & =(\a+1)^{\tau}\int_{I}\Big|u^{[\a-\frac{m+\a-2}{2}]}|u|^{\frac{m+\a-2}{2}}\partial_{x}u\Big|^{\tau}\,dx\label{eq:u_pow_est}\\
 & =\frac{4(\a+1)^{\tau}}{(m+\a)^{2}}\int_{I}|u^{\frac{\a-m+2}{2}}|^{\tau}|\partial_{x}u^{[\frac{m+\a}{2}]}|^{\tau}\,dx\nonumber \\
 & \le C(\int_{I}C_{\eta}|u^{\frac{\a-m+2}{2}}|^{\frac{2\tau}{2-\tau}}+\eta|\partial_{x}u^{\frac{m+\a}{2}}|^{2})\,dx.\nonumber 
\end{align}
Thus, since $\tau<2$ and choosing $\eta$ small enough, 
\begin{align*}
\int_{I}|u(t)|^{\a+1}dx\lesssim & \int_{I}|u_{0}|^{\a+1}dx-c\int_{0}^{t}\int_{I}(\partial_{x}u{}^{[\frac{m+\a}{2}]})^{2}dxdr\\
 & +\int_{0}^{t}\int_{I}|u|^{(\frac{\a-m+2}{2})(\frac{2\tau}{2-\tau})}dxdr+\|S\|_{W_{x}^{-1,\tau'}}^{\tau'}.
\end{align*}
Now we choose $\tau$ such that $(\frac{\a-m+2}{2})(\frac{2\tau}{2-\tau})=\a+1$, i.e.~since $m-2<\a$, 
\begin{align*}
\tau= & \frac{2\a+2}{2\a+3-m}\in(1,2].
\end{align*}
In conclusion, 
\begin{align*}
\int_{I}|u(t)|^{\a+1}dx\lesssim & \int_{I}|u_{0}|^{\a+1}dx-c\int_{0}^{t}\int_{I}(\partial_{x}u{}^{[\frac{m+\a}{2}]})^{2}dxdr\\
 & +\int_{0}^{t}\int_{I}|u|^{\a+1}dxdr+\|S\|_{W_{x}^{-1,\tau'}}^{\tau'}.
\end{align*}
Gronwall's inequality implies
\begin{align}
\int_{I}|u^{\ve}(t)|^{\a+1}dx & +\int_{0}^{t}\int_{I}(\partial_{x}(u^{\ve}){}^{[\frac{m+\a}{2}]})^{2}dxdr\lesssim\int_{I}|u_{0}|^{\a+1}dx+\|S^{\ve}\|_{W^{-1,\tau'}}^{\tau'}.\label{eq:eps-delta-bound}
\end{align}
For general initial data $u_{0}\in(L^{m+1}\cap L^{\a+1})(\R_{x}^{d})$ we choose a sequence of smooth approximations $u_{0}^{\d}\in C_{c}^{\infty}(\R_{x}^{d})$ such that $u_{0}^{\d}\to u_{0}$ in $(L^{m+1}\cap L^{\a+1})(\R_{x}^{d})$. The respective solutions $u^{\ve,\d}$ satisfy \eqref{eq:eps-delta-bound} and, due to \eqref{eq:l1-contr}, we may take the limit $\delta\to0$ to conclude. 

In order to establish \eqref{eq:kin_meas_est} we note that on the approximative level $u^{\ve,\d}$ the kinetic form is satisfied with $q^{\ve,\d}=\d_{v=u^{\ve,\d}}(\partial_{x}(u^{\ve,\d}){}^{[\frac{m+1}{2}]})^{2}$. Thus, 
\begin{align*}
\int_{[0,T]\times\bar{I}\times\R}|v|^{\a-1}q^{\ve,\d}drdxdv & =\int_{I}(\partial_{x}(u^{\ve,\d}){}^{[\frac{m+\a}{2}]})^{2}dtdx\\
 & \lesssim\int_{I}|u_{0}|^{\a+1}dx+\|S^{\ve}\|_{W^{-1,\tau'}}^{\tau'}.
\end{align*}
Passing to the limit $\d\to0$ yields \eqref{eq:kin_meas_est}. 
\end{proof}
\begin{cor}
\label{cor:ex-appr-anderson}Let $u_{0}\in L^{m+1}(I)$. Then, there is a unique entropy solution $u^{\ve}$ to \eqref{eq:deg_Anderson_approx} and $u^{\ve}$ satisfies Definition \ref{def:kinetic_sol} with 
\[
K_{1}\lesssim\|u_{0}\|_{L^{m+1}}^{m+1}+\|S\|_{B_{\infty,\infty}^{-\eta}}^{\tau}+1
\]
for some $\tau\ge2$ and some $\eta\in(\frac{1}{2},1)$. In particular, the constants $K_{1}$ in Definition \ref{def:kinetic_sol} can be chosen uniformly in $\ve$ and 
\begin{align*}
\|u^{\ve}\|_{L^{2}([0,T];H_{0}^{1}(I))}^{2}\le K_{1}.
\end{align*}
\end{cor}

\begin{proof}
We apply Lemma \ref{lem:anders_bound} with $\a\in(0,m]$.
\end{proof}
\begin{thm}
\label{thm:approx_anderson}Assume \eqref{eq:distr_ass} and let $u^{\ve}$ be the entropy solution to \eqref{eq:deg_Anderson_approx}. Then, for all $p\in[1,m)$, $s\in[0,\frac{3}{2}\frac{1}{m})$ we have
\[
u^{\ve}\in L{}^{p}([0,T];W_{loc}^{s,p}(I))
\]
with, for all $T\ge0$, $\mcO\subset\subset I$, 
\begin{align*}
\|u^{\ve}\|_{L^{p}([0,T];W^{s,p}(\mcO))} & \le C(\|u_{0}\|_{L^{m+1}(I)}^{m+1}+\|S\|_{B_{\infty,\infty}^{-\eta}}^{\tau}+1),
\end{align*}
for some $\tau\ge2$, $C$ independent of $\ve>0$ and $\eta\in(\frac{1}{2},1)$ small enough.
\end{thm}

\begin{proof}
Let $p\in[1,m)$, $s\in[0,\frac{3}{2}\frac{1}{m})$. For simplicity we drop the $\ve$ in the notation. Rewriting \eqref{eq:kinetic_anderson} we obtain, for $\eta\in(\frac{1}{2},1)$,
\begin{align}
\partial_{t}\chi & =m|v|^{m-1}\partial_{xx}\chi+\D_{x}^{\frac{\eta}{2}}\underbrace{\D_{x}^{-\frac{\eta}{2}}\chi S}_{:=g_{0}}+\D_{x}^{\frac{\eta}{2}}\partial_{v}\underbrace{\D_{x}^{-\frac{\eta}{2}}q}_{=:g_{1}}-\D_{x}^{\frac{\eta}{2}}\partial_{v}\underbrace{\D_{x}^{-\frac{\eta}{2}}\chi vS}_{=:g_{2}}\label{eq:kinetic_anderson_2}\\
 & =m|v|^{m-1}\partial_{xx}\chi+\D_{x}^{\frac{\eta}{2}}g_{0}+\D_{x}^{\frac{\eta}{2}}\partial_{v}g_{1}-\D_{x}^{\frac{\eta}{2}}\partial_{v}g_{2}\text{ on }(0,T)\times I\times\R.\nonumber 
\end{align}
An elementary computation shows $\|\chi\|_{L_{t,v}^{1}W_{_{x}}^{\eta,1}}\lesssim\|u\|_{L_{t}^{1}W_{_{x}}^{\eta,1}}$. We next use embedding results for Besov spaces \cite[Proposition 2.78]{BCD11}, estimates for the paraproduct of functions and distributions \cite[Section 4.4.3, Theorem 1]{RS96} and Corollary \ref{cor:ex-appr-anderson} to obtain, for $\d>0$ small enough, 
\begin{align}
\|g_{0}\|_{L_{t,x,v}^{1}} & =\|\D_{x}^{-\frac{\eta}{2}}\chi S\|_{L_{t,x,v}^{1}}\lesssim\|\chi S\|_{L_{t,v}^{1}B_{1,1}^{-\eta}}\lesssim\|\chi\|_{L_{t,v}^{1}B_{1,1}^{\eta+\d}}\|S\|_{B_{\infty,\infty}^{-\eta}}\label{eq:g0_bdd}\\
 & \lesssim\|u\|_{L_{t}^{1}(W_{x}^{\eta+2\d,1})}\|S\|_{B_{\infty,\infty}^{-\eta}}\lesssim\|u\|_{L_{t}^{2}(H_{0}^{1})}^{2}+\|S\|_{B_{\infty,\infty}^{-\eta}}^{2}\le K_{1}+\|S\|_{B_{\infty,\infty}^{-\eta}}^{2}.\nonumber 
\end{align}
Moreover, using the same reasoning we obtain
\begin{equation}
\||v|^{-1}g_{2}\|_{L_{t,x,v}^{1}}=\||v|^{-1}\D_{x}^{-\frac{\eta}{2}}\chi vS\|_{L_{t,x,v}^{1}}=\|\D_{x}^{-\frac{\eta}{2}}|\chi|S\|_{L_{t,x,v}^{1}}\lesssim K_{1}+\|S\|_{B_{\infty,\infty}^{-\eta}}^{2}.\label{eq:g2_bdd}
\end{equation}
We choose a cut-off function and localize \eqref{eq:kinetic_anderson_2} as in the proof of Corollary \ref{cor:par-hyp-av}. Hence, using \eqref{eq:f_lp_bound}, we may apply Lemma \ref{lem:av-iso}, with $\eta$ sufficiently close to $\frac{1}{2}$, $\a=\frac{1}{m-1}$, $\b=2$, $\l=2-2\frac{m-2+\gamma}{m-1}$ small enough by choosing $\g$ close to one, $r>1$ small enough, $p=r'$, $q=1$, $\t=\frac{1}{m}$, such that 
\begin{align*}
(1-\t)\b\frac{\a}{r}-\t(\l+\eta) & =\t(\frac{\b}{r}-\l-\eta)\\
 & =\frac{1}{m}\left(\frac{3}{2}+(\frac{2}{r}-2)+(2\frac{\g-1}{m-1})+(\frac{1}{2}-\eta)\right)>s.
\end{align*}
This yields, for all $\mcO\subset\subset I$, 
\begin{align*}
\|u\|_{L^{p}([0,T];W^{s,p}(\mcO))} & \lesssim\|\D_{x}^{-\frac{\eta}{2}}\chi S\|_{\mcM_{t,x,v}}+\|\D_{x}^{-\frac{\eta}{2}}|v|^{-\gamma}q\|_{\mcM_{t,x,v}}+\||v|^{-1}\D_{x}^{-\frac{\eta}{2}}\chi vS\|_{\mcM_{t,x,v}}\\
 & +\|f\|_{L_{t,x,v}^{r'}}+\|f\|_{L_{t,x,v}^{1}}+\|f\|_{L_{t,x}^{p}L_{v}^{1}}+1.
\end{align*}
Hence, since
\[
\|f\|_{L_{t,x,v}^{r'}}\lesssim\|f\|_{L_{t,x,v}^{1}}+1,\quad\|f\|_{L_{t,x,v}^{1}}=\|u\|_{L_{t,x}^{1}},\quad\|f\|_{L_{t,x}^{p}L_{v}^{1}}=\|u\|_{L_{t,x}^{p}}
\]
we have, using \eqref{eq:g0_bdd}, \eqref{eq:g2_bdd}, 
\begin{align*}
 & \|u\|_{L^{p}([0,T];W^{s,p}(\mcO))}\lesssim K_{1}+\|S\|_{B_{\infty,\infty}^{-\eta}}^{2}+\|u\|_{L_{t,x}^{1}}+\|u\|_{L_{t,x}^{p}}+1.
\end{align*}
In fact, \eqref{eq:kinetic_anderson_2} is not exactly of the form \eqref{eq:kinetic_isotropic}, since $g_{1}$, $g_{2}$ allow singular moments of different order, i.e.~$\gamma\in(0,1)$ for $g_{1}$, $\g=1$ for $g_{2}$. However, in the proof of Lemma \ref{lem:av-iso}, the terms involving $g_{2}$ only lead to better behaved terms than $g_{1}$ and thus may be absorbed. We next note that by the arguments of Lemma \ref{lem:anders_bound} 
\[
\|u\|_{L_{t,x}^{1}}\lesssim\|u_{0}\|_{L_{x}^{1}}+\|S\|_{W^{-1,\tau}}^{\tau}+1,\quad\|u\|_{L_{t,x}^{p}}\lesssim\|u_{0}\|_{L_{x}^{m+1}}+\|S\|_{W^{-1,\tau}}^{\tau}+1
\]
for some $\tau\ge2.$ Hence, by Corollary \ref{cor:ex-appr-anderson} we obtain 
\begin{align*}
\|u\|_{L^{p}([0,T];W^{s,p}(\mcO))}\lesssim & \|u_{0}\|_{L^{m+1}}^{m+1}+\|S\|_{B_{\infty,\infty}^{-\eta}}^{\tau}+\|u_{0}\|_{L_{x}^{1}}+\|u_{0}\|_{L_{x}^{m+1}}+\|S\|_{W^{-1,\tau}}^{\tau}+1\\
\lesssim & \|u_{0}\|_{L^{m+1}}^{m+1}+\|S\|_{B_{\infty,\infty}^{-\eta}}^{\tau}+1,
\end{align*}
for some $\tau\ge2$.
\end{proof}

\begin{proof}
[Proof of Corollary \ref{cor:reg_weak}] By Lemma \ref{lem:anders_bound} we have 
\[
\|u^{\ve}\|_{L^{2}([0,T];H_{0}^{1})}^{2}+\|\partial_{x}(u^{\ve})^{[m]}\|_{L^{2}([0,T];L^{2})}^{2}\le C.
\]
Hence, we also have $\|u^{\ve}S^{\ve}\|_{W^{-1,2}}^{2}\lesssim\|u^{\ve}\|_{W^{1,2}}^{2}\|S^{\ve}\|_{W^{-1,2}}^{2}\le C$. By \eqref{eq:deg_Anderson_approx} we obtain 
\[
\|\partial_{t}u^{\ve}\|_{L^{2}([0,T];W^{-1,2})}^{2}\le C.
\]
The Aubin-Lions compactness Lemma yields (for a subsequence)
\[
u^{\ve}\to u\quad\text{in }L^{2}([0,T];L^{2}(I)).
\]
This allows to pass to the limit in the weak form of \eqref{eq:deg_Anderson_approx}. Hence, Theorem \ref{thm:approx_anderson} finishes the proof.
\end{proof}

\appendix

\section{Truncation property and basic estimates\label{sec:Truncation-property-and}}

From \cite[Definition 2.1]{TT07} we recall the following definition.
\begin{defn}
\label{def:tuncation_prop}Let $m$ be a complex-valued Fourier multiplier. We say that $m$ has the truncation property if, for any locally supported bump function $\psi$ on $\C$ and any $1\le p<\infty$, the multiplier with symbol $\psi(\frac{m(\xi)}{\d})$ is an $L^{p}$-multiplier as well as an $\mcM_{TV}$-multiplier uniformly in $\delta>0$, that is, its $L^{p}$-multiplier norm ($\mcM_{TV}$-multiplier norm resp.) depends only on the support and $C^{l}$ size of $\psi$ (for some large $l$ that may depend on $m$) but otherwise is independent of $\delta$.
\end{defn}

We slightly deviate from the definition of the truncation property given in \cite[Definition 2.1]{TT07} since we require it to hold also for $p=1$ and on $\mcM_{TV}$. In \cite[Section 2.4]{TT07} it was shown that multipliers corresponding to parabolic-hyperbolic PDE satisfy the truncation property for $p>1$. Accordingly we extend this property to our Definition in the following example.
\begin{example}
Let
\[
m(\tau,\xi,v)=i\tau+ia(v)\cdot\xi-(\xi,b(v)\xi)
\]
for some measurable $a:\R\to\R^{d}$, $b:\R\to\mcS_{+}^{d\times d}$. Then, $m$ satisfies the truncation property uniformly in $v$.
\end{example}

\begin{proof}
Following \cite[Section 2.4]{TT07} it remains to consider the cases $p=1$ and $\mcM_{TV}$. Arguing as in \cite[Section 2.4]{TT07} we can consider the cases $m(\tau,\xi,v)=i\tau+ia(v)\cdot\xi$ and $m(\tau,\xi,v)=-(\xi,b(v)\xi)$ separately. By invariance under linear transformations, arguing again as in \cite[Section 2.4]{TT07} it is enough to consider $\psi(i\xi_{1})$, $\psi(|\xi|^{2})$. Due to \cite[Theorem 2.5.8]{G14-2} in order to prove that these are $L^{1}$-multipliers, we need to show that their inverse Fourier transforms have finite $L^{1}$ norm, which is true since $\psi$ is a bump function. Again by \cite[Theorem 2.5.8]{G14-2} an operator is an $L^{1}$-multiplier if and only if it is given by the convolution with a finite Borel measure. As such, it can be extended to a multiplier on $\mcM_{TV}$ with the same norm.
\end{proof}
We next provide a basic $L^{p}$ estimate for symbols satisfying the truncation property uniformly. The following estimate is an extension of \cite[Lemma 2.2]{TT07} by making use of regularity in the $v$ component of $f$. As pointed out in the introduction, this allows to avoid bootstrapping arguments in the applications, which is crucial, since these bootstrapping arguments do not allow to conclude a regularity of order more than one.
\begin{lem}
\label{lem:bsc_est}Assume that $m(\xi,v)$ satisfies the truncation property uniformly in $v$. Let $\vp,\phi$ be bounded, smooth functions, $\psi$ be a smooth cut-off function and $M_{\psi}$ be the Fourier multiplier with symbol $\vp(\xi)\psi\left(\frac{m(\xi,v)}{\d}\right)$. Then, for all $1<p\le2$, $\s\ge0$, $r\in(\frac{p'}{1+\s p'},p']\cap(1,\infty)$,
\begin{align*}
\|\int M_{\psi}f\phi\,dv\|_{L_{x}^{p}} & \lesssim\|f\phi\|_{L_{x}^{p}(H_{v}^{\s,p})}\sup_{\xi\in\supp\vp}|\Omega_{m}(\xi,\d)|^{\frac{1}{r}},
\end{align*}
where $\Omega_{m}(\xi,\d)=\{v\in\supp\phi:\,|m(\xi,v)|\le\d\}$. Moreover,
\begin{align*}
\|\int M_{\psi}f\phi\,dv\|_{\mcM_{TV;x}} & \lesssim\|f\phi\|_{\mcM_{TV;x}}.
\end{align*}
\end{lem}

\begin{proof}
We first consider the case $p=2$. Then
\begin{align*}
 & \|\int M_{\psi}f\phi\,dv\|_{L_{x}^{2}}\lesssim\|\int\mcF_{x}^{-1}\vp(\xi)\psi\left(\frac{m(\xi,v)}{\d}\right)\hat{f}\phi\,dv\|_{L_{x}^{2}}\\
 & =\|\int\vp(\xi)\psi\left(\frac{m(\xi,v)}{\d}\right)\hat{f}\phi\,dv\|_{L_{\xi}^{2}}\lesssim\big\|\vp(\xi)\|\psi\left(\frac{m(\xi,v)}{\d}\right)\|_{H_{v}^{-\s,2}}\|\hat{f}\phi\|_{H_{v}^{\s,2}}\big\|{}_{L_{\xi}^{2}}\\
 & \lesssim\sup_{\xi\in\supp\vp}\|\psi\left(\frac{m(\xi,v)}{\d}\right)\|_{H_{v}^{-\s,2}}\|\hat{f}\phi\|_{L_{\xi}^{2}(H_{v}^{\s,2})}.
\end{align*}
Note
\begin{align*}
\|\hat{f}\phi\|_{L_{\xi}^{2}(H_{v}^{\s,2})}^{2} & =\int\|\hat{f}\phi\|_{H_{v}^{\s,2}}^{2}d\xi=\int|(1+\Delta_{v})^{\frac{\s}{2}}\hat{f}\phi|^{2}\,dvd\xi\\
 & =\int|\mcF_{x}(1+\Delta_{v})^{\frac{\s}{2}}f\phi|^{2}\,d\xi dv=\int|(1+\Delta_{v})^{\frac{\s}{2}}f\phi|^{2}\,dxdv\\
 & =\int\|f\phi\|_{H_{v}^{\s,2}}^{2}\,dx=\|f\phi\|_{L_{x}^{2}H_{v}^{\s,2}}^{2}.
\end{align*}
By Sobolev embeddings (cf.~e.g.~\cite[Theorem 1.66]{BCD11}) we have $H_{v}^{\s,2}\hookrightarrow L_{v}^{r'}$ for all $r'\in[2,\frac{2}{1-2\s}]\cap\R$. Hence, for $r\in[\frac{2}{1+2\s},2]\cap(1,\infty)$ we have $L_{v}^{r}\hookrightarrow H_{v}^{-\s,2}$. Fix $r\in[\frac{2}{1+2\s},2]\cap(1,\infty)$ arbitrary. Then
\begin{align*}
\|\int M_{\psi}f\phi\,dv\|_{L_{x}^{2}} & \lesssim\sup_{\xi\in\supp\vp}\|\psi\left(\frac{m(\xi,v)}{\d}\right)\|_{L_{v}^{r}}\|f\phi\|_{L_{x}^{2}(H_{v}^{\s,2})}\\
 & \lesssim\sup_{\xi\in\supp\vp}|\Omega_{m}(\xi,\d)|^{\frac{1}{r}}\|f\phi\|_{L_{x}^{2}(H_{v}^{\s,2})}.
\end{align*}
This finishes the proof in case of $p=2.$

Due to the truncation property (on $L^{1}$ and $\mcM_{TV}$) uniform in $v$, we have, for all $\eta\ge1$, 
\begin{align*}
\|\int M_{\psi}f\phi\,dv\|_{L_{x}^{\eta}} & \lesssim\|f\phi\|_{L_{x,v}^{\eta}}
\end{align*}
and
\begin{align*}
\|\int M_{\psi}f\phi\,dv\|_{\mcM_{TV}} & \lesssim\|f\phi\|_{\mcM_{TV}}.
\end{align*}

We now conclude by interpolation: From the above we have that $\overline{M}_{\psi}f:=\int M_{\psi}f\phi\,dv$ is a bounded linear operator in $L(L_{x}^{2}(H_{v}^{\s,2});L_{x}^{2})\cap L(L_{x,v}^{\eta};L_{x}^{\eta})$. By complex interpolation, for $\t\in(0,1)$, $\overline{M}_{\psi}$ is a bounded linear operator in $L([L_{x}^{2}(H_{v}^{\s,2}),L_{x,v}^{\eta}]_{\t};[L_{x}^{2},L_{x}^{\eta}]_{\t}).$ Interpolation of Banach space valued $L^{p}$-spaces yields
\[
[L_{x}^{2}(H_{v}^{\s,2}),L_{x,v}^{\eta}]_{\t}=L_{x}^{\frac{2}{1+\t(\frac{2}{\eta}-1)}}([H_{v}^{\s,2},L_{v}^{\eta}]_{\t}).
\]
Next we note that, for $\eta>1$, 
\[
[H_{v}^{\s,2},L_{v}^{\eta}]_{\t}=H_{v}^{(1-\t)\s,\frac{2}{1+\t(\frac{2}{\eta}-1)}}
\]
Hence, 
\begin{align*}
[L_{x}^{2}(H_{v}^{\s,2}),L_{x,v}^{\eta}]_{\t} & \supseteq L_{x}^{\frac{2}{1+\t(\frac{2}{\eta}-1)}}(H_{v}^{(1-\t)\s,\frac{2}{1+\t(\frac{2}{\eta}-1)}})\\{}
[L_{x}^{2},L_{x}^{\eta}]_{\t} & =L^{\frac{2}{1+\t(\frac{2}{\eta}-1)}}.
\end{align*}

Let now $p\in(1,2)$. Let $\eta>1$ be such that $\t=\frac{2-p}{p}\frac{\eta}{2-\eta}\in(0,1)$, i.e.~$p=\frac{2}{1+\t(\frac{2}{\eta}-1)}$. Then, in conclusion, for all $\s>0$ and all $r\in[\frac{2}{1+2\s},2]\cap(1,\infty)$,
\begin{align*}
\|\int M_{\psi}f\phi\,dv\|_{L^{p}} & =\|\int M_{\psi}f\phi\,dv\|_{L^{\frac{2}{1+\t(\frac{2}{\eta}-1)}}}\\
 & \lesssim\|\overline{M}_{\psi}\|_{L(L_{x}^{2}(H_{v}^{\s,2});L_{x}^{2})}^{1-\t}\|\overline{M}_{\psi}\|_{L(L_{x,v}^{\eta};L_{x}^{\eta})}^{\t}\|f\phi\|_{L_{x}^{\frac{2}{1+\t(\frac{2}{\eta}-1)}}(H_{v}^{(1-\t)\s,\frac{2}{1+\t(\frac{2}{\eta}-1)}})}\\
 & \lesssim\sup_{\xi}|\Omega_{m}(\xi,\d)|^{\frac{2}{rp'}}\|f\phi\|_{L_{x}^{p}(H_{v}^{2\s\frac{p-\eta}{p(2-\eta)},p})}.
\end{align*}

Now given $\s>0$ we apply the above with $\s$ replaced by $\s':=\frac{p(2-\eta)}{2(p-\eta)}\s>0$ and $\eta>1$ small enough. Again choosing $\eta>1$ small enough, this yields the claim for all $r\in(\frac{p'}{1+\s p'},p']\cap(1,\infty)$. 
\end{proof}

\section{\label{app:kin_solutions}Entropy solutions for parabolic-hyperbolic PDE with a source}

In this section we present a sketch of the proof of well-posedness of entropy/kinetic solutions for PDE of the type
\begin{equation}
\partial_{t}u+\div A(u)=\div\left(b(u)\nabla u\right)+S(t,x)\quad\text{on }(0,T)\times\R^{d}\label{eq:app-par-hyp}
\end{equation}
with 
\begin{align}
u_{0} & \in L^{1}(\R^{d}),\,S\in L^{1}([0,T]\times\R^{d})\nonumber \\
a & =A'\in L_{loc}^{\infty}(\R;\R^{d})\label{eq:ph-as-1-1}\\
b_{ij}(\cdot) & =\sum_{k=1}^{d}\s_{ik}(\cdot)\s_{kj}(\cdot),\quad\s_{ik}\in L_{loc}^{\infty}(\R;\R^{d}).\nonumber 
\end{align}

\begin{thm}
\label{thm:wp-kinetic}Let $u_{0}\in L^{1}(\R^{d})$, $S\in L^{1}([0,T]\times\R^{d}).$ Then there is a unique entropy solution $u$ to \eqref{eq:app-par-hyp} satisfying $u\in C([0,T];L^{1}(\R^{d}))$. For two entropy solutions $u^{1}$, $u^{2}$ with initial conditions $u_{0}^{1},u_{0}^{2}$ and forcing $S^{1},S^{2}$ we have 
\begin{equation}
\sup_{t\in[0,T]}\|u^{1}(t)-u^{2}(t)\|_{L^{1}(\R^{d})}\le\|u_{0}^{1}-u_{0}^{2}\|_{L^{1}(\R^{d})}+\|S^{1}-S^{2}\|_{L^{1}([0,T]\times\R^{d})}.\label{eq:l1-contr}
\end{equation}
Moreover, if $u_{0}\in L^{p}(\R^{d})$, $S\in L^{p}([0,T]\times\R^{d})$ for some $p\in[1,\infty)$, then
\begin{align}
\sup_{t\in[0,T]}\|u(t)\|_{L_{x}^{p}} & \le C(\|u_{0}\|_{L_{x}^{p}}+\|S\|_{L_{t,x}^{p}}),\label{eq:lp-bound}
\end{align}
for some constant $C=C(T,p)$.
\end{thm}

\begin{proof}
\red{\textbf{Uniqueness:} We present a sketch of the proof. The argument is a combination of \cite{CP03,GL17} and is rigorously justified following the convolution error estimates from \cite{CP03,GL17}. Owing to \cite[proof of Theorem 11]{GL17} we note that $g(t,x,v)=1_{v<u(t,x)}$ satisfies the same kinetic equation as $f$. We further note that, informally, due to Definition \ref{def:kinetic_sol-1}, (ii), (iii),
\[
n(t,x,v)=\delta_{v=u(t,x)}\sum_{k=1}^{d}\left(\sum_{i=1}^{d}\partial_{x_{i}}\b_{ik}(u(t,x))\right)^{2}.
\]

We next note that,
\begin{align*}
\partial_{t}\int_{\R}g^{1}(1-g^{2})\,dv= & \int_{\R}\partial_{t}g^{1}(1-g^{2})-g^{1}\partial_{t}g^{2}\,dv\\
= & \int_{\R}(-a(v)\cdot\nabla_{x}g^{1}+\div(b(v)\nabla_{x}g^{1})+\partial_{v}q^{1}+\delta_{v=u^{1}}S^{1})(1-g^{2})\\
 & -g^{1}(-a(v)\cdot\nabla_{x}g^{2}+\div(b(v)\nabla_{x}g^{2})+\partial_{v}q^{2}+\delta_{v=u^{2}}S^{2})\,dv\\
= & -\div_{x}\int_{\R}(a(v)g^{1}(1-g^{2})\,dv+2\int_{\R}\nabla_{x}g^{1}\cdot b(v)\nabla_{x}g^{2}\,dv\\
 & +\int_{\R}(q^{1}\partial_{v}g^{2}+\partial_{v}g^{1}q^{2})\,dv+\int(\delta_{v=u^{1}}S^{1})(1-g^{2})-g^{1}(\delta_{v=u^{2}}S^{2})\,dv.
\end{align*}
Concerning the forcing, as in \cite{GL17}, we observe that
\begin{align*}
\int_{\R}(\delta_{v=u^{1}}S^{1})(1-g^{2})-g^{1}(\delta_{v=u^{2}}S^{2})\,dv & =1_{u^{1}\ge u^{2}}(S^{1}-S^{2}).
\end{align*}
Next, as in \cite{CP03},
\begin{align*}
 & \int_{\R}(q^{1}\partial_{v}g^{2}+\partial_{v}g^{1}q^{2})\,dv=-\int_{\R}(q^{1}\delta_{v=u^{2}}+\delta_{v=u^{1}}q^{2})\,dv\\
 & \le-\int_{\R}\sum_{k=1}^{d}\left(\sum_{i=1}^{d}\partial_{x_{i}}\b_{ik}(u^{1})\right)^{2}\delta_{v=u^{1}}\delta_{v=u^{2}}+\delta_{v=u^{1}}\delta_{v=u^{2}}\sum_{k=1}^{d}\left(\sum_{i=1}^{d}\partial_{x_{i}}\b_{ik}(u^{2})\right)^{2}\,dv\\
 & \le-2\int_{\R}\delta_{v=u^{1}}\delta_{v=u^{2}}\sum_{k=1}^{d}\left(\sum_{i=1}^{d}\partial_{x_{i}}\b_{ik}(u^{1})\sum_{j=1}^{d}\partial_{x_{j}}\b_{jk}(u^{2})\right)\,dv\\
 & =-2\int_{\R}\delta_{v=u^{1}}\delta_{v=u^{2}}\sum_{i,j,k=1}^{d}\s_{ik}(u^{1})\partial_{x_{i}}u^{1}\s_{jk}(u^{2})\partial_{x_{j}}u^{2}\,dv\\
 & =-2\int_{\R}\delta_{v=u^{1}}\delta_{v=u^{2}}\sum_{i,j=1}^{d}b_{ij}(v)\partial_{x_{i}}u^{1}\partial_{x_{j}}u^{2}\,dv.
\end{align*}
Note that, informally (justified as in \cite{CP03} based on the chain-rule Definition \ref{def:kinetic_sol-1} (ii)), 
\[
2\int_{\R}\nabla_{x}g^{1}\cdot b(v)\nabla_{x}g^{2}\,dv=2\sum_{i,j=1}^{d}\int_{\R}b_{ij}(v)\delta_{v=u^{1}}\partial_{x_{i}}u^{1}\delta_{v=u^{2}}\partial_{x_{j}}u^{2}\,dv.
\]
We thus obtain that
\begin{align*}
\partial_{t}\int_{\R^{d+1}}g^{1}(1-g^{2})\,dvdx & \le\int_{\R^{d}}1_{u^{1}\ge u^{2}}(S^{1}-S^{2})\,dx.
\end{align*}
Since $\int g^{1}(1-g^{2})\,dvdx=\int(u^{1}-u^{2})_{+}\,dx$ this implies
\begin{align*}
\int_{\R^{d}}(u^{1}(t)-u^{2}(t))_{+}\,dx & \le\int_{\R^{d}}(u_{0}^{1}-u_{0}^{2})_{+}\,dx+\int_{0}^{t}\int_{\R^{d}}1_{u^{1}\ge u^{2}}(S^{1}-S^{2})\,dxdr,
\end{align*}
which by reversing the roles of $u^{1}$ and $u^{2}$ implies \eqref{eq:l1-contr}.

\textbf{Existence:} Step 1: Assume that $u_{0}\in C_{c}^{\infty}(\R_{x}^{d})$, $S\in C_{c}^{\infty}((0,T)\times\R_{x}^{d})$.

The construction of solutions relies on a smooth, non-degenerate approximation of $A$, $b$. Let $A^{\ve}:\R\to\R^{d}$, $b^{\ve}:\R\to\mcS_{+}^{d\times d}$ be smooth, Lipschitz continuous, satisfying $b^{\ve}(u)\ge\ve Id$ for all $u\in\R$, $\ve>0$ and $A^{\ve},b^{\ve}\to A,b$ locally uniformly. Then, by \cite{LSU67} there is a classical solution to
\begin{equation}
\partial_{t}u^{\ve}+\div A^{\ve}(u^{\ve})=\div\left(b^{\ve}(u^{\ve})\nabla u^{\ve}\right)+S(t,x)\quad\text{on }(0,T)\times\R_{x}^{d}.\label{eq:visc_approx}
\end{equation}
For $\eta\in C^{2}(\R_{v})$ convex we have 
\begin{align}
\partial_{t}\int_{\R_{x}^{d}}\eta(u^{\ve}(t))dx= & \int_{\R_{x}^{d}}\eta'(u^{\ve}(t))(-\div A^{\ve}(u^{\ve})+\div\left(b^{\ve}(u^{\ve})\nabla u^{\ve}\right)+S(t,x))dx\nonumber \\
= & \int_{\R_{x}^{d}}-\eta'(u^{\ve}(t))(A^{\ve})'(u^{\ve})\cdot\nabla u^{\ve}-\eta''(u^{\ve}(t))(\nabla u^{\ve}\cdot b^{\ve}(u^{\ve})\nabla u^{\ve})\label{eq:eta-bound}\\
 & +\eta'(u^{\ve}(t))S(t,x)\,dx\nonumber \\
\le & \int_{\R_{x}^{d}}\eta'(u^{\ve}(t))S(t,x)\,dx.\nonumber 
\end{align}
Hence, by a standard approximation argument, for all $p\in[1,\infty)$,
\begin{align*}
\frac{1}{p}\partial_{t}\int_{\R_{x}^{d}}|u^{\ve}(t)|^{p}dx & \le\int_{\R_{x}^{d}}u^{\ve}(t)^{[p-1]}S(t,x)dx\lesssim\int_{\R_{x}^{d}}|u^{\ve}(t)|^{p}+|S(t,x)|^{p}\,dx
\end{align*}
and thus
\begin{align}
\sup_{t\in[0,T]}\|u^{\ve}(t)\|_{L_{x}^{p}} & \le C(\|u_{0}\|_{L_{x}^{p}}+\|S\|_{L_{t,x}^{p}}).\label{eq:approx-lp-bound}
\end{align}
By the $L^{1}$-contraction \eqref{eq:l1-contr} we further have, uniformly in $\ve>0$,
\begin{align*}
\sup_{t\in[0,T]}\|u^{\ve}\|_{\dot{BV}_{x}}\le & \|u_{0}\|_{\dot{BV}_{x}}+\sup_{t\in[0,T]}\|S\|_{\dot{BV}_{x}}\\
\|\partial_{t}u^{\ve}(t,\cdot)\|_{L_{x}^{1}}\le & \|\partial_{t}u^{\ve}(0)\|_{L_{x}^{1}}+\|\partial_{t}S\|_{L_{t,x}^{1}}\\
\le & \|\div A(u_{0})+\div(b(u_{0})\nabla u_{0})+S(0,\cdot)\|_{L_{x}^{1}}+\|\partial_{t}S\|_{L_{t,x}^{1}}.
\end{align*}
Since $u^{\ve}$ is a classical solution it is easy to verify that $u^{\ve}$ is an entropy solution following the lines of \cite[Section 7]{CP03}. The above estimates imply the convergence (of a non-relabeled subsequence) $u^{\ve}\to u$ in $C([0,T];L^{1}(\R^{d}))$. The verification that $u$ is an entropy solution again follows from the same arguments as \cite[Section 7]{CP03}. The $L^{p}$ bound \eqref{eq:lp-bound} follows from \eqref{eq:approx-lp-bound}.

Step 2: Let now $u_{0}\in L^{1}(\R_{x}^{d})$, $S\in L^{1}((0,T)\times\R_{x}^{d})$. 

We choose $u_{0}^{\ve}\in C_{c}^{\infty}(\R^{d})$, $S^{\ve}\in C_{c}^{\infty}((0,T)\times\R^{d}))$ such that $u_{0}^{\ve}\to u_{0}$ in $L^{1}(\R_{x}^{d})$ and $S^{\ve}\to S$ in $L^{1}((0,T)\times\R^{d}))$. By the $L^{1}$-contraction \eqref{eq:l1-contr} this implies that $u^{\ve}\to u$ in $C([0,T];L^{1}(\R^{d}))$. The verification that $u$ is an entropy solutions again follows along the lines of \cite[Section 7, Step 2]{CP03}.}
\end{proof}

\section{\label{app:ebm}The case $m\ge2$}

In this section we present an improvement of the results obtained in \cite{E05}. We consider 
\begin{equation}
\partial_{t}u+\div A(u)=\D u^{[m]}+S(t,x)\quad\text{on }(0,T)\times\R_{x}^{d}\label{eq:par-hyp-1}
\end{equation}
where
\begin{align}
u_{0} & \in L^{1}(\R_{x}^{d}),\,S\in L^{1}([0,T]\times\R_{x}^{d})\nonumber \\
a=A' & \in L_{loc}^{\infty}(\R;\R^{d}),\label{eq:ph-as-1}\\
u^{[m]}=|u|^{m-1}u & \text{ with }m\ge2.\nonumber 
\end{align}

By \cite{CP03} and Appendix \ref{app:kin_solutions} there is a unique entropy solution to \eqref{eq:par-hyp-1}.
\begin{lem}
\label{lem:ebm}Let $\gamma>0$, $u_{0}\in(L^{1}\cap L^{1+\g})(\R_{x}^{d})$ and $S\in(L^{1}\cap L^{1+\g})([0,T]\times\R_{x}^{d})$. Then, there are $c_{\gamma,m},C_{\gamma}>0$ such that
\begin{align*}
\sup_{t\in[0,T]}\|u(t)\|_{1+\gamma}^{1+\gamma}+c_{\gamma,m}\int_{0}^{T}\int_{\R_{x}^{d}}(\nabla u{}^{[\frac{\gamma+m}{2}]})^{2}dx & \le C_{\gamma}(\|u_{0}\|_{L_{x}^{1+\gamma}}^{1+\gamma}+\|S\|_{L_{t,x}^{1+\gamma}}^{1+\gamma}).
\end{align*}
\end{lem}

\begin{proof}
First let $u_{0}\in C_{c}^{\infty}(\R_{x}^{d})$, $S\in C_{c}^{\infty}((0,T)\times\R_{x}^{d})$ and $A^{\ve}$ be smooth, Lipschitz continuous with $A^{\ve}\to A$ locally uniformly. Then, for $\ve>0$, there is a unique classical solution to
\[
\partial_{t}u^{\ve}+\div A^{\ve}(u^{\ve})=\ve\D u^{\ve}+\D(u^{\ve})^{[m]}+S(t,x)\quad\text{on }(0,T)\times\R_{x}^{d}.
\]
From \eqref{eq:eta-bound} we have, for $\eta\in C^{2}(\R)$ convex, Lipschitz continuous, 
\begin{align*}
\partial_{t}\int_{\R_{x}^{d}}\eta(u^{\ve}(t))dx\le & \int_{\R_{x}^{d}}-\eta''(u^{\ve}(t))|\nabla u^{\ve}(t)|^{2}|u^{\ve}(t)|^{m-1}+\eta'(u^{\ve}(t))S(t,x)dx\\
\le & \int_{\R_{x}^{d}}-|\nabla F^{\eta}(u^{\ve}(t))|^{2}+\frac{\gamma}{1+\gamma}|\eta'(u^{\ve}(t))|^{\frac{1+\gamma}{\gamma}}+\frac{1}{1+\gamma}|S(t,x)|^{1+\gamma}dx,
\end{align*}
where $F^{\eta}(u):=m\int_{0}^{u}\sqrt{\eta''(r)|r|^{m-1}}dr$. Integrating in time and choosing a suitable approximation of $\eta$ this inequality may be applied to $\eta(u)=|u|^{1+\gamma}$, which yields 
\begin{align*}
\int_{\R_{x}^{d}}|u^{\ve}(t)|^{1+\gamma}dx\lesssim & \int_{\R_{x}^{d}}|u_{0}|^{1+\gamma}dx-\frac{4\gamma m(1+\gamma)}{(\gamma+m)^{2}}\int_{0}^{t}\int_{\R_{x}^{d}}(\nabla(u^{\ve}){}^{[\frac{\gamma+m}{2}]})^{2}\,dx\\
 & +\int_{\R_{x}^{d}}|u^{\ve}|^{1+\gamma}+|S(t,x)|^{1+\gamma}\,dx.
\end{align*}
Gronwall's inequality yields
\begin{align}
\sup_{t\in[0,T]}\|u^{\ve}(t)\|_{L_{x}^{1+\gamma}}^{1+\gamma}+c_{\gamma,m}\int_{0}^{T}\int_{\R_{x}^{d}}(\nabla(u^{\ve}){}^{[\frac{\gamma+m}{2}]})^{2}dxds & \le C_{\gamma}(\|u_{0}\|_{L_{x}^{1+\gamma}}^{1+\gamma}+\|S\|_{L_{t,x}^{1+\gamma}}^{1+\gamma}).\label{eq:eps-bound}
\end{align}
From the construction of entropy solutions (Theorem \ref{thm:wp-kinetic}) we have $u^{\ve}\to u$ in $C([0,T];L^{1}(\R_{x}^{d}))$. Moreover, by \eqref{eq:eps-bound} for a non-relabeled subsequence $\nabla(u^{\ve}){}^{[\frac{\gamma+m}{2}]}\rightharpoonup Z$ in $L^{2}([0,T]\times\R_{x}^{d})$. Since $u^{\ve}\to u$ a.e.~we have $Z=\nabla(u){}^{[\frac{\gamma+m}{2}]}$ which allows to pass to the limit in \eqref{eq:eps-bound}.

For general $u_{0}\in(L^{1}\cap L^{1+\g})(\R_{x}^{d})$, $S\in(L^{1}\cap L^{1+\g})([0,T]\times\R_{x}^{d})$ we choose smooth, compactly supported approximations $u_{0}^{\ve}$, $S^{\ve}$ with $\|u_{0}^{\ve}\|_{L_{x}^{1+\gamma}}^{1+\gamma}\le\|u_{0}\|_{L_{x}^{1+\gamma}}^{1+\gamma}$ and $\|S^{\ve}\|_{L_{t,x}^{1+\gamma}}^{1+\gamma}\le\|S\|_{L_{t,x}^{1+\gamma}}^{1+\gamma}$ and $u_{0}^{\ve}\to u_{0}$, $S^{\ve}\to S$ in $L^{1}$. The corresponding entropy solution $u^{\ve}$ then satisfies \eqref{eq:eps-bound}. Due to Theorem \ref{thm:wp-kinetic} we have $u^{\ve}\to u$ in $C([0,T];L^{1}(\R_{x}^{d}))$ which allows to pass to the limit in \eqref{eq:eps-bound} as above. 
\end{proof}
For $p\in[1,\infty)$, $s\in(0,1)$ we recall
\[
\|f\|_{\dot{\mcN}^{s,p}}^{p}:=\sup_{\d>0}\sup_{0<|z|<\d}\int_{\R_{x}^{d}}\left|\frac{|f(x+z)-f(x)|}{|z|^{s}}\right|^{p}\,dx
\]
and
\[
\|f\|{}_{\mcN^{s,p}}^{p}=\|f\|_{L^{p}}^{p}+\|f\|_{\dot{\mcN}^{s,p}}^{p}.
\]

\begin{thm}
\label{thm:ebm}Let $\gamma>0$, $m\ge2$ and $u_{0}\in(L^{1}\cap L^{1+\gamma})(\R_{x}^{d})$, $S\in(L^{1}\cap L^{1+\gamma})([0,T]\times\R_{x}^{d})$. Then
\[
\|u\|_{L^{m+\g}([0,T];\dot{\mcN}^{\frac{2}{m+\g},m+\g}(\R_{x}^{d}))}^{m+\gamma}\le C_{\gamma,m}(\|u_{0}\|_{L_{x}^{1+\gamma}}^{1+\gamma}+\|S\|_{L_{t,x}^{1+\gamma}}^{1+\gamma}).
\]
If, in addition, $u_{0}\in L^{m+\gamma}(\R_{x}^{d})$, $S\in L^{m+\gamma}([0,T]\times\R_{x}^{d})$ then $u\in L^{m+\g}([0,T];\mcN^{\frac{2}{m+\g},m+\g}(\R_{x}^{d}))$ with 
\begin{equation}
\|u\|{}_{L^{m+\g}([0,T];\mcN^{\frac{2}{m+\g},m+\g}(\R_{x}^{d}))}^{m+\g}\le C_{\gamma,m}(\|u_{0}\|_{L_{x}^{1+\gamma}}^{1+\gamma}+\|S\|_{L_{t,x}^{1+\gamma}}^{1+\gamma}+\|u_{0}\|_{L_{x}^{m+\gamma}}^{m+\gamma}+\|S\|_{L_{t,x}^{m+\gamma}}^{m+\gamma}).\label{eq:eb}
\end{equation}
\end{thm}

\begin{proof}
We again restrict to giving the informal derivation, the rigorous justification is standard by considering a vanishing viscosity approximation first, then using lower semicontinuity. From \cite[Lemma 4.1]{E05} we recall the elementary inequality, for $m\ge2$, 
\[
|r-s|^{m}\le c|r^{[\frac{m}{2}]}-s^{[\frac{m}{2}]}|^{2}\quad\forall r,s\in\R,
\]
for some $c>0$. Hence,
\begin{align*}
|\D_{e}^{h}u(x)|^{m} & =|u(x+he)-u(x)|^{m}\le c|u(x+he)^{[\frac{m}{2}]}-u(x)^{[\frac{m}{2}]}|^{2}\\
 & =c|\D_{e}^{h}u^{[\frac{m}{2}]}(x)|^{2}
\end{align*}
and thus, using Lemma \ref{lem:ebm},
\begin{align*}
 & \int_{0}^{T}\sup_{h>0}\sup_{e\in\R^{d},|e|=1}\int_{\R^{d}}\left|\frac{\D_{e}^{h}u(t,x)}{h^{\frac{2}{m+\g}}}\right|^{m+\g}dxdt\\
 & =\int_{0}^{T}\sup_{h>0}\sup_{e\in\R^{d},|e|=1}\int_{\R^{d}}h^{-2}\left|\D_{e}^{h}u(t,x)\right|^{m+\g}dxdt\\
 & \le c\int_{0}^{T}\sup_{h>0}\sup_{e\in\R^{d},|e|=1}\int_{\R^{d}}h^{-2}|\D_{e}^{h}u^{[\frac{m+\g}{2}]}(t,x)|^{2}dxdt\\
 & \le c\int_{0}^{T}\int_{\R^{d}}|\nabla u^{[\frac{m+\g}{2}]}(t,x)|^{2}dxdt\\
 & \le C_{\gamma,m}(\|u_{0}\|_{L_{x}^{1+\gamma}}^{1+\gamma}+\|S\|_{L_{t,x}^{1+\gamma}}^{1+\gamma}).
\end{align*}
This implies
\begin{align*}
\int_{0}^{T}\|u(t,\cdot)\|_{\dot{\mcN}^{\frac{2}{m+\g},m+\g}(\R_{x}^{d})}^{m+\g}\,dt & \le C_{\gamma,m}(\|u_{0}\|_{L_{x}^{1+\gamma}}^{1+\gamma}+\|S\|_{L_{t,x}^{1+\gamma}}^{1+\gamma}).
\end{align*}

Using Lemma \ref{lem:ebm} with $\g$ replaced by $m-1+\gamma$ yields
\[
\|u\|_{L^{\infty}([0,T];L^{m+\g}(\R_{x}^{d}))}^{m+\g}\le C_{m,\g}(\|u_{0}\|_{L_{x}^{m+\g}}^{m+\g}+\|S\|_{L_{t,x}^{m+\gamma}}^{m+\gamma}).
\]

This implies that
\[
\|u\|_{L^{m+\g}([0,T];\mcN^{\frac{2}{m+\g},m+\g}(\R^{d}))}^{m+\g}\le C_{\gamma}(\|u_{0}\|_{L_{x}^{1+\gamma}}^{1+\gamma}+\|S\|_{L_{t,x}^{1+\gamma}}^{1+\gamma})+C_{m,\g}(\|u_{0}\|_{L_{x}^{m+\g}}^{m+\g}+\|S\|_{L_{t,x}^{m+\gamma}}^{m+\gamma}).
\]
\end{proof}

\section{\label{app:optimality_scaling}Optimality and scaling}

In this section we present scaling arguments that indicate the optimal regularity of solutions of porous medium equations. We then show that these estimates are indeed sharp since they are attained by the Barenblatt solution. Consider 
\begin{align}
\partial_{t}u & =\D(|u|^{m-1}u)\quad\text{on }(0,T)\times\R_{x}^{d}\label{eq:elliptic-1}\\
u(0) & =u_{0}\quad\text{on }\R_{x}^{d},\nonumber 
\end{align}
with $u_{0}\in L^{1}(\R_{x}^{d})$, $m>1$. 
\begin{lem}
Assume that for some $s\ge0$, $p\ge1$, $C\ge0$ we have
\begin{equation}
\|u\|_{L^{p}([0,T];\dot{W}^{s,p}(\R_{x}^{d}))}^{p}\le C\|u_{0}\|_{L^{1}(\R_{x}^{d})},\label{eq:eb-2-1}
\end{equation}
for all solutions $u$ to \eqref{eq:elliptic-1}. Then, necessarily $p\le m$ and $s\le\frac{p-1}{p}\frac{2}{m-1}\le\frac{2}{m}.$ 
\end{lem}

\begin{proof}
Given a solution $u$ to \eqref{eq:elliptic-1}, for every $\eta>0$, also $\td u(t,x):=u(\eta t,x)\eta^{\frac{1}{m-1}}$ is a solution to \eqref{eq:elliptic-1}. Since $\|\td u\|_{L^{p}([0,T];\dot{W}^{s,p}(\R_{x}^{d}))}^{p}=\eta^{\frac{p}{m-1}-1}\|u\|_{L^{p}([0,\eta T];\dot{W}^{s,p}(\R_{x}^{d}))}^{p}$ and $\|\td u(0)\|_{L^{1}(\R_{x}^{d})}=\eta^{\frac{1}{m-1}}\|u_{0}\|_{L^{1}(\R_{x}^{d})}$ from \eqref{eq:eb-2-1} we obtain that
\[
\|u\|_{L^{p}([0,T];\dot{W}^{s,p}(\R_{x}^{d}))}^{p}\le C\eta^{1-\frac{p-1}{m-1}}\|u_{0}\|_{L^{1}(\R_{x}^{d})}.
\]
This leads to a contradiction (letting $\eta\uparrow\infty$), unless 
\begin{equation}
p\le m.\label{eq:nec1}
\end{equation}

Similarly, we may rescale in space: Given a solution $u$ to \eqref{eq:elliptic-1}, for every $\eta>0$, also $\td u(t,x):=u(t,\eta x)\eta^{-\frac{2}{m-1}}$ is a solution to \eqref{eq:elliptic-1}. Note that $\|\td u(0)\|_{L^{1}(\R_{x}^{d})}=\eta^{-\frac{2}{m-1}-d}\|u_{0}\|_{L^{1}(\R_{x}^{d})}$ and $\|\td u\|_{L^{p}([0,T];\dot{W}^{s,p}(\R_{x}^{d}))}^{p}=\eta^{-\frac{2}{m-1}p+ps-d}\|u\|_{L^{p}([0,T];\dot{W}^{s,p}(\R_{x}^{d}))}^{p}$ . Hence, by \eqref{eq:eb-2-1},
\[
\|u\|_{L^{p}([0,T];\dot{W}^{s,p}(\R_{x}^{d}))}^{p}\le C\eta^{\frac{2}{m-1}(p-1)-ps}\|u_{0}\|_{L^{1}(\R_{x}^{d})},
\]
which leads to a contradiction unless $s\le\frac{p-1}{p}\frac{2}{m-1}.$ Maximizing the right hand side under \eqref{eq:nec1} yields $p=m$ and $s\le\frac{2}{m}.$
\end{proof}
\begin{example}
Consider the Barenblatt solution 
\begin{align*}
u(t,x) & =t^{-\a}(C-k|xt^{-\b}|^{2})_{+}^{\frac{1}{m-1}},
\end{align*}
where $m>1$, $\a=\frac{d}{d(m-1)+2}$, $k=\frac{\a(m-1)}{2md}$ , $\b=\frac{\a}{d}$ and $C>0$ is a free constant. Then 
\[
u\in L^{m}([0,T];\dot{W}^{s,m}(\R_{x}^{d}))
\]
implies $s<\frac{2}{m}.$
\end{example}

\begin{proof}
With $F(x)=(C-k|x|^{2})_{+}^{\frac{1}{m-1}}$ we have $u(t,x)=t^{-\a}F(xt^{-\b}).$ We next observe that, for $s\in(0,1)$,
\begin{align*}
\|u(t,\cdot)\|_{\dot{W}^{s,m}(\R_{x}^{d})}^{m} & =\int_{\R_{x}^{d}}\int_{\R_{y}^{d}}\frac{|u(t,x)-u(t,y)|^{m}}{|x-y|^{sm+d}}dxdy\\
 & =t^{-\a m-\b(sm+d)+2d\b}\|F\|_{\dot{W}^{s,m}(\R_{x}^{d})}^{m}.
\end{align*}
Hence,
\begin{align*}
\|u\|_{L^{m}([0,T];\dot{W}^{s,m}(\R_{x}^{d}))}^{m} & =\|t^{-\a m-\b(sm+d)+2d\b}\|_{L^{1}([0,T])}\|F\|_{\dot{W}^{s,m}(\R_{x}^{d})}^{m}.
\end{align*}
which is finite if and only if
\[
-\a m-\b(sm+d)+2d\b>-1\quad\text{and}\quad F\in\dot{W}^{s,m}(\R_{x}^{d}).
\]
Hence, necessarily
\[
-m-\frac{1}{d}(sm+d)+2>-\frac{1}{\a}=-\left(\frac{d(m-1)+2}{d}\right),
\]
which is equivalent to $2>ms.$ In the case $s\in(1,2)$ we observe that $\partial_{x_{i}}u(t,x)=t^{-(\a+\b)}F_{x_{i}}(xt^{-\b})$ so that analogous arguments may be applied.
\end{proof}
\bibliographystyle{abbrv}
\bibliography{refs}

\end{document}